\documentclass[12pt]{amsart}
\usepackage{amssymb,amsfonts,amsmath,amsopn,amstext,amscd,color,latexsym,mathabx,mathrsfs,skak,stackrel,xy,url,verbatim}

\input xy
\xyoption{all}

\theoremstyle{remark}

\newtheorem{para}{\bf}[section]

\theoremstyle{definition}

\newtheorem{dfn}[para]{Definition}

\theoremstyle{plain}

\newtheorem{thm}[para]{Theorem}

\newtheorem{lemma}[para]{Lemma}
\newtheorem{cor}[para]{Corollary}
\newtheorem{prop}[para]{Proposition}

\newenvironment{numequation}{\addtocounter{para}{1}
\begin{equation}}{\end{equation}}

\newcommand{\bbG}{{\mathbb G}}

\newcommand{\bbN}{{\mathbb N}}

\newcommand{\bbP}{{\mathbb P}}
\newcommand{\bbQ}{{\mathbb Q}}

\newcommand{\bbZ}{{\mathbb Z}}

\newcommand{\bB}{{\bf B}}

\newcommand{\bG}{{\bf G}}

\newcommand{\bL}{{\bf L}}
\newcommand{\bM}{{\bf M}}

\newcommand{\bP}{{\bf P}}

\newcommand{\bT}{{\bf T}}
\newcommand{\bU}{{\bf U}}

\newcommand{\fra}{{\mathfrak a}}
\newcommand{\frb}{{\mathfrak b}}

\newcommand{\frg}{{\mathfrak g}}

\newcommand{\frl}{{\mathfrak l}}

\newcommand{\frp}{{\mathfrak p}}
\newcommand{\frq}{{\mathfrak q}}

\newcommand{\frt}{{\mathfrak t}}
\newcommand{\fru}{{\mathfrak u}}

\newcommand{\frx}{{\mathfrak x}}

\newcommand{\cD}{{\mathcal D}}

\newcommand{\cF}{{\mathcal F}}
\newcommand{\cG}{{\mathcal G}}

\newcommand{\cL}{{\mathcal L}}

\newcommand{\cO}{{\mathcal O}}

\newcommand{\cX}{{\mathcal X}}

\newcommand{\cZ}{{\mathcal Z}}

\newcommand{\sL}{{\mathscr L}}

\newcommand{\sR}{{\mathscr R}}

\newcommand{\Qp}{{\mathbb Q_p}}

\newcommand{\Ext}{{\rm Ext}}

\newcommand{\Hom}{{\rm Hom}}

\newcommand{\Lie}{{\rm{Lie}}}

\newcommand{\Spec}{{\rm Spec}}

\newcommand{\alg}{{\rm alg}}

\newcommand{\lra}{\longrightarrow}

\newcommand{\midc}{{\; | \;}}

\newcommand{\ra}{\rightarrow}

\newcommand{\sub}{\subset}

\newcommand{\GL}{{\rm GL}}

\newcommand{\Sym}{{\rm Sym}}

\newcommand{\car}{\stackrel{\simeq}{\longrightarrow}}

\newcommand{\hU}{\hat{U}}
\newcommand{\dG}{\dim}

\newcommand{\dg}{\dim_{U(\frg)}}

\newcommand{\jGn}{j_{D(G_0)}}\newcommand{\jGnr}{j_{D_r(G_0)}}
\newcommand{\jU}{j_{U(\frg)}}\newcommand{\jUr}{j_{U_r(\frg)}}

\newcommand{\Oa}{{\cO_\alg}}
\newcommand{\Ad}{{\rm Ad}}

\theoremstyle{definition}
\newtheorem{exam}[para]{\bf Example}

\begin{document}

\title{Dimensions of some locally analytic representations}
\author{Tobias Schmidt}
\address{Humboldt-Universit\"at zu Berlin, Institut f\"ur Mathematik, Rudower Chaussee 25, 12489 Berlin, Germany}
\email{Tobias.Schmidt@math.hu-berlin.de}
\author{Matthias Strauch}
\address{Indiana University, Department of Mathematics, Rawles Hall, Bloomington, IN 47405, U.S.A.}
\email{mstrauch@indiana.edu}

\thanks{T.S. would like to acknowledge support of the Heisenberg programme of Deutsche Forschungsgemeinschaft. M. S. would like to acknowledge the support of the National Science Foundation (award DMS-1202303).}

\begin{abstract} Let $G$ be the group of points of a split reductive group over a finite extension of $\bbQ_p$. In this paper, we compute the dimensions of certain classes of locally analytic $G$-representations. This includes principal series representations and certain representations coming from homogeneous line bundles on $p$-adic symmetric spaces. As an application, we compute the dimensions in Colmez' unitary principal series of ${\rm GL}_2(\bbQ_p)$.
\end{abstract}

\maketitle

\tableofcontents

\section{Introduction}

Let $L$ be a finite extension of $\Qp$ and let $G = \bG(L)$ be the group of $L$-valued points of a split connected reductive algebraic group $\bG$ over $L$. Let $P\subseteq G$ be a parabolic subgroup.

\vskip5pt

Admissible Banach space representations or locally analytic representations of $G$ admit a well-behaved notion of (canonical) dimension. The rational representations coming from the algebraic group $\bG$ or the traditional smooth representation from Langlands theory are known to have dimension zero. Moreover, any representation which is not zero-dimensional has dimension greater or equal to half the dimension of the minimal nilpotent orbit of $\bG$ \cite{AW}, \cite{SchmidtDIM}. Besides these general results, the dimensions of even very explicit representations like principal series representations have not been computed so far. In this paper, we make an attempt to close this gap and determine the dimensions of certain families of representations. This includes principal series representations as well as representations coming from $p$-adic symmetric spaces. The technical key result is that the functor $\cF_P^G$ introduced by Orlik and the second author in \cite{OrlikStrauchJH} from Lie algebra representations of $\frg$ endowed with a compatible action of $P$ to locally analytic $G$-representations preserves the dimension.

\vskip5pt

As an application, we compute the dimensions in Colmez' unitary principal series of ${\rm GL}_2(\bbQ_p)$ \cite{Colmez_trianguline}. Let $\Pi(V)$ denote the unitary representation associated by Colmez' $p$-adic local Langlands correspondence \cite{Colmez10} to an absolutely irreducible $2$-dimensional $p$-adic Galois representations $V$ of $Gal(\bar{\bbQ}_p/\bbQ_p)$. The resulting map

$$ V\mapsto {\rm dim}\;\Pi(V)$$ 

\vskip8pt 

is bounded above by the number $2$ due to the presence of infinitesimal characters. We show that its restriction to trianguline representations is constant with value $1$. This raises the question whether there are absolutely irreducible $2$-dimensional representations $V$ of $Gal(\bar{\bbQ}_p/\bbQ_p)$ with the property $\dG \Pi(V)=2$.

\vskip5pt

In the following we give more details on the individual sections of this paper.
In section $2$ we review basic notions of dimension theory and establish two auxiliary lemmas. In section $3$ we develop a framework which allows us to prove faithful flatness of Arens-Michael envelopes in many situations. In section $4$ we combine this result, in the case of the universal enveloping algebra $U(\frg)$ of $\frg=Lie(G)$, with a study of the functor $\cF_P^G$ and prove that the latter preserves dimensions. On the level of Lie algebra representations, canonical dimension coincides with the more traditional Gelfand-Kirillov dimension and this enables us to give explicit dimension formulas for the representations $\cF_P^G(M)$ whenever the Gelfand-Kirillov dimension for $M$ (viewed as an $U(\frg)$-module) is known. We illustrate this in section $5$ in the case of the classical parabolic Bernstein-Gelfand-Gelfand category for $\frp\subseteq\frg$ where $\frp=Lie(P)$. For example, the dimension of the locally analytic parabolic induction ${\rm Ind}_P^G(V)$ where $V$ is a locally analytic $P$-representations on a finite-dimensional vector space equals the vector space dimension of $\frg/\frp$. We also remark that the dimensions of irreducible objects in the $BGG$-category can be computed out of the Kazhdan-Lusztig conjecture through Joseph's Goldie rank polynomials. The main result of \cite{OrlikStrauchJH} shows that the functor $\cF_P^G$ preserves irreducibility in many cases which yields the dimensions of all the irreducible $G$-representations which can be constructed through a functor of type $\cF_P^G$. In section $6$ we let $G=\GL_{d+1}(L)$ and compute the dimension of locally analytic representations coming from homogeneous line bundles on Drinfeld's upper half space \cite{Orlik08}. In section $7$ we give the aforementioned application to ${\rm GL}_2(\bbQ_p)$ and its unitary principal series.

\vskip16pt

{\it Notation and conventions:} We denote by $p$ a prime number and consider fields $L \sub K$ which are both finite extensions of $\Qp$.
Let $o_L$ and $o_K$ be the rings of integers of $L$, resp. $K$, and let $|\cdot |_K$ be the absolute value on $K$ such that $|p|_K = p^{-1}$. The field $L$ is our ''base field'', whereas we consider $K$ as our ''coefficient field''. For a locally convex $K$-vector space $V$ we denote by $V'_b$ its strong dual, i.e., the $K$-vector space of continuous linear forms equipped with the strong topology of bounded convergence. Sometimes, in particular when $V$ is finite-dimensional, we simplify notation and write $V'$ instead of $V'_b$. All finite-dimensional $K$-vector spaces are equipped with the unique Hausdorff locally convex topology.

\vskip5pt

We let $\bG$ be a split reductive group scheme over $o_L$ and $\bT \sub \bB \sub \bG$ a maximal split torus and a Borel subgroup scheme, respectively. We denote the base change to $L$ of these group schemes by the same letters. We let $\bB\subseteq \bP$ be a parabolic subgroup and let $\bL_\bP$ the unique Levi subgroup which contains $\bT$. By $G_0 = \bG(o_L)$, $B = \bB(o_L)$, etc., and $G = \bG(L)$, $B = \bB(L)$, etc., we denote the corresponding groups of $o_L$-valued points and $L$-valued points, respectively. Finally, Gothic letters $\frg$, $\frp$, etc., will denote the Lie algebras of $\bG$, $\bP$, etc.: $\frg = \Lie(\bG)$, $\frt = \Lie(\bT)$, $\frb = \Lie(\bB)$, $\frp = \Lie(\bP)$,
$\frl_{\frp} = \Lie(\bL_\bP)$, etc.. Base change to $K$ is usually denoted by the subscript ${}_K$, for instance, $\frg_K = \frg \otimes_L K$.

\section{Grade and dimension}\label{sec-dim}
In this section we introduce some basic notions in dimension theory and establish two simple lemmas.
The term module always means {\it left} module. Noetherian rings are two-sided noetherian and other ring-theoretic
properties are used similarly.

\vskip5pt

We recall the notion of an Auslander
regular ring \cite{LVO}. Let $R$ be an arbitrary associative unital ring. For any $R$-module $N$ the {\it grade} $j_R(N)$ is
defined to be either the smallest integer $k$ such that
Ext$_R^k(N,R)\neq 0$ or $\infty$. Now suppose that $R$ is (left and right) noetherian. If $N\neq 0$ is finitely generated, then its
grade $j_R(N)$ is bounded above by the projective dimension of $N$. A noetherian ring
$R$ is called {\it Auslander regular} if its global dimension is finite and if every finitely
generated $R$-module $N$ satisfies {\it Auslander's
condition}: for any $k\geq 0$ and any $R$-submodule
$L\subseteq$\,Ext$_R^k(N,R)$ one has $j_R(L)\geq k$.

\vskip5pt

Let $R$ be an Auslander regular ring of finite global dimension $gld(R)$ and $M$ an
$R$-module. The number

$$\dim_RM:=gld(R)-j_R(M)$$

\vskip8pt

is called the {\it canonical dimension} of $M$. One has 

$$\dim_R\;M=\max\{\dim_RM',\dim_RM''\}$$ 

\vskip8pt  

for an exact sequence $0\rightarrow M'\rightarrow M\rightarrow M''\rightarrow 0$. Moreover, $\dim_R0=-\infty$.

\vskip5pt

Let $\tau$ be an automorphism of $R$ and let $M$ be a left $R$-module. We denote by $^{\tau} \hskip-3pt M$ the abelian group $M$ with the left $R$-action $r.m:=\tau(r)m$ and call $^\tau \hskip-3pt M$ the {\it twist of $M$ with $\tau$}. In case of a right module $M$ we denote the analogous construction by $M^\tau$. 

\begin{lemma}\label{lemma-twisting}
Twisting with $\tau$ has the following properties:
\begin{itemize}
\item[(i)] the functor $M\rightarrow \hskip3pt ^{\tau}\hskip-3pt M$ is an auto-equivalence on the category of all $R$-modules,
\item[(ii)] $M$ is finitely generated if and only if $^{\tau} \hskip-3pt M$ is finitely generated,
\item[(iii)] there are canonical isomorphisms $\Ext_R^{k}(^{\tau} \hskip-3pt M,R)\simeq \Ext_R^{k}(M,R)^\tau$ for all $k$,
\item[(iv)] one has $j_R(M)=j_R(^{\tau} \hskip-3pt M)$.
\end{itemize}
\end{lemma}
\begin{proof}
Twisting with $\tau^{-1}$ yields a quasi-inverse, so (i) is clear. (ii) is trivial so let us turn to (iii). In case of $k=0$ the isomorphism is given explicitly by sending a linear form $f$ on $^\tau \hskip-3pt M$ to the linear form $\tau\circ f$ on $M$. According to (i) a projective resolution $P_\bullet$ for $M$ yields a projective resolution $^\tau \hskip-3pt P$ for $^\tau \hskip-3pt M$. Since the isomorphism for $k=0$ is natural in $M$, we are done. (iv) follows formally from (iii).
\end{proof}

 \begin{lemma}\label{lemma-basechangegrade}
Let $R\rightarrow S$ be a faithfully flat ring extension between noetherian rings. Let $M$ be a finitely generated $R$-module and put $M_S:=S\otimes_R M$. We have

$$j_{S}(M_S)=j_{R}(M) \;.$$ 

\vskip8pt 
\end{lemma}

\begin{proof}
We have $\Ext_{R}^k(M,R)\otimes_R S \simeq \Ext_{S}^k(M_S,S)$ for all $k$. Indeed, since $R\rightarrow S$ is flat, choosing a free resolution of $M$ by finitely generated free modules reduces us to the case $k=0$ and $M=R$ where the statement is obvious. By faithful flatness of $R\rightarrow S$, we have $\Ext_{R}^{j_{R}(M)}(M,R)\otimes_R S\neq 0$ which implies the claim.
\end{proof}

\section{Arens-Michael envelopes and faithful flatness}

Let $R$ be a complete discrete valuation ring with field of fractions $K$ and uniformizer $\pi$. Let $A$ be an $R$-algebra, flat as an $R$-module,
equipped with an increasing and exhaustive filtration 

$$F_0A \subseteq F_1A\subseteq F_2A\subseteq \ldots$$ 

\vskip8pt 

by $R$-submodules such that $1\in F_0A$ and $F_iA \cdot F_j{A}\subseteq F_{i+j}A$ for all $i,j$. In particular, $F_0A$ is an $R$-subalgebra of $A$.
We make the following three assumptions on this filtration.
\begin{itemize}
\item[(1)] We have $F_i A\cdot F_j A = F_j A\cdot F_i A$ as $R$-submodules of $A$ for all $i,j$;
\item[(2)] the ring $F_0A$ is a commutative noetherian integral domain such that $F_0A/\pi F_0A$ is a regular integral domain;
\item[(3)] the associated graded ring $gr_\bullet^F A$ is commutative and isomorphic to a polynomial ring over $F_0A$ in finitely many, say $r$, variables (where the polynomial ring has its usual positive grading by total degree with the variables placed in degree one).
\end{itemize}
The regularity assumption in (2) means that all local rings of $F_0A/\pi F_0A$ at prime ideals are regular or, equivalently, that the ring $F_0A/\pi F_0A$ has finite global dimension. Of course, any filtration with $F_0A=R$ satisfies (2), but there is no point in restricting to this special case at the moment. 

\vskip5pt

Positively filtered algebras $A$ that satisfy these requirements abound. The main examples we have in mind are universal enveloping algebras of Lie algebras as well as the rings of (crystalline) differential operators on certain smooth affine $R$-schemes. We will give more details at the end of this section.

\vskip5pt
In the following we will assume that these conditions hold. We then have the $K$-algebras

$$(F_0A)_K:=F_0A\otimes_{R} K \hskip10pt {\rm and} \hskip10pt A_K:=A \otimes_{R} K \;.$$ 

\vskip8pt 

The algebra $(F_0A)_K$ has a natural structure of normed algebra by declaring the lattice $F_0A$ to be the unit ball. We give $A_K$ the finest locally convex topology making the inclusion map $(F_0A)_K\hookrightarrow A_K$ continuous (where the source has its norm topology).
Our aim in this subsection is to analyze the algebraic and homological properties of the Arens-Michael envelope $\hat{A}_K$ of the locally convex algebra $A_K$. Recall \cite{Helemskii}\footnote{The classical notion of Arens-Michael envelope \cite{Helemskii} is given for complex algebras, but the definition extends readily to any valued field.} that 

$$\hat{A}_K:= {\rm (Hausdorff)~completion~ of}~A_K {\rm ~w.~r.~t.~ all~ continuous~ submultiplicative~ seminorms.}$$ 

\vskip8pt 

Among our main results will be that $\hat{A}_K$ is a Fr\'echet-Stein algebra in the sense of \cite{ST03} and that the canonical completion homomorphism $A_K\ra \hat{A}_K$ is a faithfully flat ring extension. As we will see, these results make the homological algebra of $\hat{A}_K$ quite transparent.

\vskip5pt

As a first step we will obtain a more accessible description of $\hat{A}_K$. To this end, we consider the Rees ring 

$$R_\bullet^F(A):= \oplus_{i\geq 0} (F_{i}A)X^{i}$$ 

\vskip8pt  

of the filtered ring $A$, viewed as a subring of the polynomial ring $A[X]$. The ring $R_\bullet^F(A)$ is noetherian according to \cite[II.2.2.1]{LVO}. For each number $n\geq 0$ we let $A_n$ be the image of $R_\bullet^F(A)$ under the evaluation homomorphism $A[X]\ra A$ given by $X\mapsto \pi^{n}$. Obviously, $A_{n+1} \subseteq A_n$ and $A_0=A$.
Let $\hat{A}_{n}$ be the $\pi$-adic completion of $A_n$ and put $\hat{A}_{n,K}:=\hat{A}_n\otimes_R K$. All rings $A_n, \hat{A}_n$ and $\hat{A}_{n,K}$ are noetherian.

\vskip5pt

In the following we will need some basic results on the interplay between the positive filtration $F_\bullet A$ on $A$, the $\pi$-adic filtration on $A$ and the rings $A_n$. Such results are established by K. Ardakov and S. Wadsley in \cite{AW} and to be completely clear, we therefore relate our situation to the terminology used in loc.cit. The positively filtered ring $A$ is an almost commutative $R$-algebra in the sense of the definition \cite[3.4]{AW}. Moreover, it is deformable and $A_n$ is its $n$-th deformation \cite[3.5]{AW}. According to \cite[Prop. 3.8]{AW} the algebra $\hat{A}_{n,K}$ is therefore an almost commutative affinoid $K$-algebra in the sense of the definition \cite[3.8]{AW}. In particular, $\hat{A}_{n,K}$ is a complete doubly filtered $K$-algebra with slice $\hat{A}_n/\pi \hat{A}_n$ \cite[3.1]{AW}. Of course, we have $A_n/\pi A_n=\hat{A}_n/\pi \hat{A}_n$.

\vskip5pt

Each ring $A_n$ has its induced filtration $F_m A_n:= A_n \cap F_mA$.
Since $gr_\bullet^F A$ is flat over $R$ one has

\begin{numequation}\label{induced_filtration}F_m A_n = \sum_{i=0,...,m} \pi^{in}F_i A \;.\end{numequation}
 
In particular, $F_0A_n=F_0A$. The graded ring $gr_\bullet^F A_n$ is in fact isomorphic to the graded ring $gr_\bullet^F A$ via the map given on the $i$-th homogeneous component as

\begin{numequation}\label{comparison_graded} F_i A/F_{i-1}A\lra F_i A_n/F_{i-1}A_n,\hskip10pt x+ F_{i-1}A \mapsto \pi^{in}x+ F_{i-1}A_n \;,
\end{numequation}

\cite[Lem. 3.5]{AW}. In particular, $gr_\bullet^F A_n$ is isomorphic, as a graded ring, to a polynomial ring in $r$ variables over $F_0A_n$.

\vskip5pt

The slice $A_n/\pi A_n$ has the quotient filtration coming from $F_\bullet A_n$. We let $gr_\bullet^{F} (A_n/\pi A_n)$ be the associated graded ring.
According to \cite[Lem. 3.7]{AW} the map $A_n\rightarrow A_n/\pi A_n$ induces an isomorphism of graded rings

\begin{numequation}\label{lem-AW}
gr^F_\bullet A_n/\pi gr^F_\bullet A_n\stackrel{\simeq}{\longrightarrow}gr_\bullet^{F} (A_n/\pi A_n) \;.
\end{numequation}

Following \cite[3.1]{AW} we finally abbreviate

$$ {\rm Gr}(\hat{A}_{n,K}):=gr_\bullet^{F} (A_n/\pi A_n) \;.$$ 

\vskip8pt 

This is a polynomial ring over $F_0A_n/\pi F_0A_n$ in $r$ variables and is therefore a noetherian regular integral domain according to (2).

\begin{prop}\label{prop-FS}
The homomorphism  $\hat{A}_{n+1,K}\ra\hat{A}_{n,K}$ is flat for all $n$.
\end{prop}
\begin{proof}
We follow an overall strategy of Berthelot \cite[3.5.3]{BerthelotDI} which is made explicit in \cite[5.3.10]{EmertonA}. As a starting point, we equip the ring $A_n$ with the following 'augmented' filtration:

$$F'_m A_n:= A_{n+1}\cdot F_m A_n$$ 

\vskip8pt 

for all $m$. We claim that this filtration satisfies 

$$F'_k A_n\cdot F'_\ell A_n \subseteq F'_{k+\ell}A_n$$ 

\vskip8pt  

for all $k,\ell$ so that we have an associated graded ring $gr_\bullet^{F'} A_n$. To prove the claim, it suffices to verify

$$A_{n+1}\cdot F_m A_n = F_m A_n \cdot A_{n+1} \;.$$
 
\vskip8pt

Because of $A_{n+1}=\sum_{j\geq 0} \pi^{(n+1)j}F_j A$ together with (\ref{induced_filtration}) this reduces to

$$ \pi^{(n+1)j}F_j A\cdot \pi^{in}F_i A =\pi^{in} F_i A\cdot \pi^{(n+1)j}F_j A$$ 

\vskip8pt

for each $i,j$. However, this is a direct consequence of our hypothesis (1). Secondly, we observe that $F_0 A_n= F_0 A$ which implies $gr_0^{F'}A_n= F'_0 A_n= A_{n+1}$. Finally, we claim that the ring $gr_\bullet^{F'}A_n$ is finitely generated over $gr_0^{F'} A_n$ by central elements. To start with, the composite
 $$F_{m+1} A_n\subseteq F'_{m+1} A_n\ra F'_{m+1} A_n/ F'_m A_n$$ \vskip8pt  is surjective and factors through $F_{m+1} A_n/F_m A_n$ for all $m \geq 0$.
We obtain a graded ring homomorphism 
 
$$f: gr^F_\bullet A_n\ra gr^{F'}_\bullet A_n$$ 

\vskip8pt  

whose image equals $F_0A_n\oplus (\oplus_{m>0} gr_m^{F'}A_n)$ with $F_0A_n\subset gr^{F'}_0 A_n$. According to the isomorphism (\ref{comparison_graded}) and our hypothesis (3) on $gr_\bullet^F A$, the source of $f$ is a polynomial ring over $F_0A_n$ in finitely many variables, say $y_1,...,y_r\in gr_1^{F} A_n$. It therefore suffices to see that the images of these generators in $gr^{F'}_\bullet A_n$ are central, that is, they commute with $gr^{F'}_0 A_n=F'_0 A_n=A_{n+1}$. To this end, we choose elements $x_1,...,x_r$ in $F_1 A$ such that $y_i=\pi^n x_i + F_0 A_n$. This is possible according to (\ref{comparison_graded}). The commutator $[gr_0^{F'}A_n,f(y_i)]$ vanishes in $gr^{F'}_\bullet A_n$, if we can show the inclusion $[A_{n+1},\pi^nx_i+F_0 A_n]\subseteq A_{n+1}$ inside $A_n$. Since $F_0A_n=F_0A_{n+1}\subset A_{n+1}$ and since $[\cdot,\pi^nx_i]$ is additive, we are reduced to show

$$[ \pi^{(n+1)j}z, \pi^nx_i] \in A_{n+1}$$ 

\vskip8pt  

for any $z\in F_j A$ and $j\geq 0$. Since $gr^F_\bullet A$ is commutative, the commutator $[z,x_i]\in F_{j+1}A$ lies in fact in the subgroup $F_j A$. This implies

$$ [ \pi^{(n+1)j}z, \pi^nx_i]=\pi^{(n+1)j+n}[z,x_i]\in \pi^n\cdot \pi^{(n+1)j}F_j A\subset \pi^n\cdot F_j A_{n+1}\subset F_j A_{n+1}$$ 

\vskip8pt 
 
which proves the claim. All in all, we have now verified the conditions (i),(ii),(iii) appearing in \cite[Lem. 5.3.9]{EmertonA} for the augmented filtration $F'_\bullet A_n$ and its subring $F'_0 A_n=A_{n+1}$. Hence, \cite[Prop. 5.3.10]{EmertonA} implies the flatness of $\hat{A}_{n+1,K}\ra\hat{A}_{n,K}$.
 \end{proof}

The proposition implies that the projective limit  

$$\varprojlim_n\hat{A}_{n,K} \;,$$ 

\vskip8pt 
 
with its projective limit topology, is a Fr\'echet-Stein algebra in the sense of \cite{ST03}.

\begin{prop}\label{prop-AM}
 The canonical map $\varprojlim_n\hat{A}_{n,K}\stackrel{\simeq}{\rightarrow} \hat{A}_K$ is an isomorphism of topological algebras.
 \end{prop}
 
\begin{proof}
Clearly, any $A_n$ gives rise to a continuous submultiplicative seminorm, say $||.||_n$, on $A_K$ and it suffices to see that these are cofinal in the directed set of all such seminorms on $A_K$. According to \cite[Lem. 3.1]{AW}, the graded ring of $A_n$ relative to its $\pi$-adic filtration is isomorphic to a polynomial ring $(A_n/\pi A_n)[t]$ in one variable $t$ over $A_n/\pi A_n$. Since ${\rm Gr}(\hat{A}_{n,K})$ is an integral domain, the rings $A_n/\pi A_n$ and $(A_n/\pi A_n)[t]$ are integral domains, too. This implies that $||.||_n$ is in fact multiplicative. After these preliminaries, we consider an arbitrary continuous and submultiplicative seminorm $||.||$ on $A_K$. Choose a graded isomorphism between $gr_\bullet^{F} A$ and a polynomial ring over $F_0A$ and lift the variables to elements $x_1,...,x_r$ in $F_1A$. By (2) the ring $F_0A$ is an integral domain and, hence, so is $gr_\bullet^{F}A$. In particular, the principal symbol map for $gr_\bullet^{F}A$ is multiplicative. It follows that the ordered monomials
 $\underline{x}^{\underline{k}}:=x_1^{k_1}\cdot\cdot\cdot x_r^{k_r}$ for $\underline{k}:=(k_1,...,k_r)\in\bbN^r$ form a basis of the $F_0A$-module $A$. Take an element $a\in A_K$ and write
 
$$ a=\sum_{\underline{k}}a_{\underline{k}}\underline{x}^{\underline{k}}$$ 

\vskip8pt 
 
with uniquely determined $a_{\underline{k}}\in (F_0A)_K$. Let $|.|$ be the norm on $(F_0A)_K$ and choose $n$ large enough such that $||x_i||\leq |\pi|^{-n}$ for all $i$. By (\ref{comparison_graded}) the symbols of the elements $\pi^nx_i$ in $gr_\bullet^{F} A_n$ are in degree one and constitute a complete set of variables over $F_0A_n$. Repeating the argument above for $A_n$ shows that $||\pi^n x_i||_n=1$ for all $i$ and that

$$||a||_n=\max_{\underline{k}} |a_{\underline{k}}| \cdot \prod_i ||x_i||_n^{k_i}
=\max_{\underline{k}} |a_{\underline{k}}| \cdot \prod_i |\pi|^{-nk_i} \;.$$ 

\vskip8pt 

Our assertion follows now from 

$$||a||\leq \max_{\underline{k}} |a_{\underline{k}}| \cdot ||\underline{x}^{\underline{k}}||\leq \max_{\underline{k}} |a_{\underline{k}}| \prod_i ||x_i||^{k_i}\leq ||a||_n \;.$$ 

\end{proof}

\vskip8pt
 
\begin{prop}\label{prop-ff}
  The canonical homomorphism $A_K\ra \hat{A}_K$ is faithfully flat.
\end{prop}

Before we turn to the proof of the proposition we establish two auxiliary lemmas.
We consider the $\pi$-adic filtration on $A_n, \hat{A}_n $ and $\hat{A}_{n,K}$. Let $gr^\pi_\bullet A_n$ be the associated graded ring of $A_n$ and let $t$ be the principal symbol of $\pi$. Of course, $gr^\pi_\bullet A_n=gr^\pi_\bullet \hat{A}_n$. As we have explained above $gr^\pi_\bullet A_n= (A_n/\pi A_n)[t]$ equals the polynomial ring over $A_n/\pi A_n$ in the variable $t$. In particular, 

$$gr^\pi_\bullet \hat{A}_{n,K}=(A_n/\pi A_n)[t^{\pm 1}] \;.$$ 

\vskip8pt  

Since ${\rm Gr}(\hat{A}_{n,K})$ is noetherian, the ring $A_n/\pi A_n$ is noetherian, too. So $gr^\pi_\bullet A_n$ is noetherian. Since $A_n$ is $R$-flat, the $\pi$-adic filtration on $A_n$ is separated. Since $\pi$ is a central and regular element in $A_n$, we have the Artin-Rees property for the $\pi$-adic filtration on $A_n$ \cite[Cor. I.4.4.8]{LVO}. This implies that the Rees ring associated with the $\pi$-adic filtration of $A_n$ is noetherian \cite[Thm. II.1.1.5]{LVO} and this finally allows us to apply the theory of lifted Ore sets as explained in \cite{LiHuishi}. To do this, let $T_n\subseteq gr^\pi_\bullet A_n$ be the central and multiplicative subset equal to $\{1,t,t^2,...\}$ and put $$S_n:= \{s\in A_n: \sigma(s)\in T\}.$$ \vskip8pt  Here, $\sigma$ denotes the principal symbol map for the $\pi$-adic filtration on $A_n$. One has $S_n= \{ \pi^m(1+I_n): m\geq 0\}$ where $I_n$ denotes the ideal of $A_n$ generated by $\pi$. Recall the notion of an Ore set in a (noncommutative) ring \cite[2.1.13]{MCR}.

\begin{lemma} The set $S_n$ is an Ore set in $A_n$. There is a filtration on the localization $S_n^{-1}A_n$ making $A_n\ra S_n^{-1}A_n$ a filtered homomorphism. The associated graded ring is canonically isomorphic to the localization $T_n^{-1}(gr^\pi_\bullet A_n)$. The completion homomorphism $S_n^{-1}A_n\rightarrow \widehat{S_n^{-1}A}$ is faithfully flat.
\end{lemma}

\begin{proof}
The statements about the Ore set, the filtration and the graded ring follow from \cite[Cor. 2.2/Cor. 2.4]{LiHuishi}. Note that the filtration on $S_n^{-1}A_n$ is Zariskian in the sense of \cite{LVO} and therefore $S_n^{-1}A_n\ra \widehat{S_n^{-1}A_n}$ is indeed faithfully flat \cite[Thm. II.2.1.2]{LVO}.
\end{proof}

\begin{lemma}
In the situation of the preceding lemma, the canonical homomorphism $A_n\rightarrow \hat{A}_{n,K}$ extends to an isomorphism of $K$-algebras
$$\widehat{S_n^{-1}A_n}\car \hat{A}_{n,K}.$$ \vskip8pt 
\end{lemma}

\begin{proof}
The canonical homomorphism $h: A_n\ra \hat{A}_{n,K}$ is of course filtered relative to $\pi$-adic filtrations. Moreover, $h(1+I_n)$ consists of units in $\hat{A}_n$ which implies $h(s)\in (\hat{A}_{n,K})^\times$ for each $s\in S_n$. For any $m$ we denote the homogeneous component of $gr^\pi_\bullet A_n$ of degree $m$ by $gr^\pi_m A_n$, and similarly for the graded rings $gr^\pi_\bullet \hat{A}$ and $gr^\pi_\bullet \hat{A}_{n,K}$. Given $s\in S_n$ with $\sigma(s)\in gr^\pi_m A_n$ we have $\sigma(h(s))\in gr^\pi_m \hat{A}_{n,K}$.
We have already explained that $A_n/\pi A_n$ is an integral domain. Hence, the graded ring $gr^\pi_\bullet \hat{A}_{n,K}=(A_n/\pi A_n)[t^{\pm 1}]$ is an integral domain, too, and therefore its principal symbol map is multiplicative. Since $\sigma(1)=1\in gr_0 \hat{A}_{n,K}$, we deduce that $\sigma (h(s)^{-1})\in gr_{-m} \hat{A}_{n,K}$. The universal property of microlocalization \cite[Prop. IV.1.1.3]{LVO} applied to $h$ therefore yields a filtered homomorphism

$$\hat{h}: \widehat{S_n^{-1}A_n}\ra\hat{A}_{n,K}$$ 

\vskip8pt  

such that $h=\hat{h}\circ q$ where $q$ equals the canonical map $A_n\ra\widehat{S_n^{-1}A_n}$. We claim that $\hat{h}$ is an isomorphism. Since the filtrations on source and target are exhaustive, separated and complete, it suffices to check that its graded map is an isomorphism \cite[Cor. I.4.2.5]{LVO}. However
 this graded map equals the canonical map between the graded ring of $S_n^{-1}A_n$ and $gr^\pi_\bullet \hat{A}_{n,K}=T_n^{-1}(gr^\pi_\bullet A_n)$ which is an isomorphism according to the preceding lemma.
\end{proof}

We now turn to the proof of the proposition.

\begin{proof}
Consider a (left) ideal $J \subset A_K$. Since $A_K=A_n\otimes_{o_K} K\ra \hat{A}_{n,K}$ is flat, the map \begin{numequation}\label{flatness_n}\hat{A}_{n,K} \otimes_{A_K} J\lra \hat{A}_{n,K}\end{numequation} is injective for any $n$. The ring $A_K$ being noetherian, the $A_K$-module $J$ is finitely presented and, hence, so is the $\hat{A}_K$-module $\hat{A}_K\otimes_{A_K} J$. It is therefore a coadmissible module for the Fr\'echet-Stein algebra $\hat{A}_K$ \cite[Cor. 3.4]{ST03} and, consequently, equals the projective limit over the modules $\hat{A}_{n,K} \otimes_{A_K} J$. Since the projective limit is left-exact, we obtain thereby from (\ref{flatness_n}) the injectivity of the map $\hat{A}_{K} \otimes_{A_K} J\ra \hat{A}_{K}$. This establishes the flatness of the map $A_K\ra \hat{A}_K$.

We turn to faithful flatness.  To this end, consider a (left) $A_K$-module $M$ and assume $\hat{A}_K \otimes_{A_K} M=0$. Since $A_K\ra\hat{A}_K$ is flat, we may assume \cite[3.3.5]{BerthelotDI} that $M$ is a cyclic module on one generator, say $m$. According to the first lemma, the completion homomorphism $S_n^{-1}A_n\ra \widehat{S_n^{-1}A_n}$ is faithfully flat. Moreover, $\pi\in S_n$, so that $S_n^{-1}A_n=S_n^{-1}A_K$. According to the second lemma, we have an isomorphism $\widehat{S_n^{-1}A_n}\simeq \hat{A}_{n,K}$. We may therefore deduce from 

$$\widehat{S_n^{-1}A_n}\otimes_{S_n^{-1}A_n} S_n^{-1}M=\hat{A}_{n,K} \otimes_{A_K} M=\hat{A}_{n,K}\otimes_{\hat{A}_K} (\hat{A}_K \otimes_{A_K} M) = 0$$ 

\vskip8pt  

that $S_n^{-1}M=0$, in other words, $M$ is $S_n$-torsion for all $n$. Thus, there exists an element $f_n\in S_n$ with $f_nm=0$ for all $n$. However, $S_n$ is of the form $\cup_{m\geq 0} \pi^m\cdot (1+I_n)$ where $I_n$ denotes the ideal generated by $\pi$ in $A_n=\sum_{i\geq 0} \pi^{ni} F_iA$. Since $\hat{A}_{n,K}$ is $\pi$-adically complete, the elements in $1+\pi F_0A$ are units in $\hat{A}_{n,K}$ which allows us to assume that $f_n$ is of the form $1+\pi^ng_n$ with some element $g_n\in A$. Since $Am$ is contained in the $K$-vector space $M$, it is $R$-free and hence $\pi$-adically separated. The limit of the sequence $f_nm\in Am$ in the $\pi$-adic topology equals $m$. Thus, $m=0$ and $M=0$. This completes the proof of the proposition.
\end{proof}
  
  \begin{cor}\label{cor-basechangegrade}
Let $M$ be a finitely generated $A_K$-module and $\hat{M}:=\hat{A}_K\otimes_{A_K} M$. Then 
   
$$j_{A_K}(M)=j_{\hat{A}_K}( \hat{M}) \;.$$ 

\vskip8pt 
\end{cor}

\begin{proof}
Use Prop. \ref{prop-ff} and Lem. \ref{lemma-basechangegrade}.
\end{proof}

\vskip5pt

We have already explained that the ring ${\rm Gr}(\hat{A}_{n,K})$ is a noetherian regular integral domain. Let $d$ denote its (finite) global dimension. Of course, $d$ equals the sum of the global dimension of $F_0A/\pi F_0A$ and the number $r$ as defined in (3).

\begin{prop}\label{prop-AR} The noetherian ring
$\hat{A}_{n,K}$ is Auslander regular of global dimension $\leq d$.
\end{prop}
\begin{proof}
The ring ${\rm Gr}(\hat{A}_{n,K})$ is Auslander regular \cite[III.2.4.3]{LVO} and therefore $A_n/\pi A_n$ is Auslander regular of global dimension $\leq d$ according to \cite[II.3.1.4]{LVO} and \cite[III.2.2.5]{LVO}. According to \cite[III.3.4.6]{LVO} we obtain that the rings $gr^\pi_\bullet \hat{A}_{n}=(A_n/\pi A_n)[t]$ and $gr^\pi_\bullet \hat{A}_{n,K}=(A_n/\pi A_n)[t^{\pm 1}]$ are Auslander regular of global dimension $\leq d+1$. A second application of \cite[II.3.1.4]{LVO} and \cite[III.2.2.5]{LVO} now yields that the rings $\hat{A}_{n}$ and $\hat{A}_{n,K}$ are Auslander regular of global dimension $\leq d+1$. On the other hand, $\pi\hat{A}_n$ is contained in the Jacobson radical of $\hat{A}_n$ according to \cite[I.3.3.5]{LVO} and so $\pi$ annihilates any simple $\hat{A}_n$-module. Hence the global dimension of $\hat{A}_{n,K}=\hat{A}_n[\pi^{-1}]$ is even $\leq d$ by \cite[7.4.3/7.4.4]{MCR}.
\end{proof}

According to the proposition the Fr\'echet-Stein algebra $\hat{A}_K$ verifies that assumption {\rm (DIM)} as formulated in \cite[8.8]{ST03}. Consequently, the grade number $j_{\hat{A}_K}$ is a well-behaved codimension function on the abelian category of coadmissible modules. This implies the following corollary, cf. \cite[Lem. 8.4]{ST03}.

\begin{cor}\label{cor-mingrade} If $M$ is a coadmissible $\hat{A}_K$-module and $M_n:= \hat{A}_{n,K}\otimes_{\hat{A}_K} M$, then 

$$j_{\hat{A}_K}(M)=\min_n j_{\hat{A}_{n,K}}(M_n) \;.$$ 

\vskip8pt 
\end{cor}

\vskip5pt

We finish with a discussion of examples of algebras $A$ satisfying our requirements.
Let $\frg$ be a $R$-Lie algebra which is finite and free as an $R$-module, say of rank $d$. Let $A:=U(\frg)$ be its universal enveloping algebra equipped with its usual positive filtration. Then $A$ satisfies all our requirements. Indeed, (1) follows by definition of the filtration and (2) is trivial since $F_0A=R$. It is well-known that the graded ring of $U(\frg)$ equals the symmetric algebra of the $R$-module $\frg$ whence (3). Note that $A_K=U(\frg_K)$ with the $K$-Lie algebra $\frg_K:=\frg\otimes_R K$ and that $\hat{A}_{n,K}=\hU(\pi^{n}\frg)_K$, i.e. $\hat{A}_{n,K}$ coincides with the $\pi$-adic completion with subsequent inversion of $\pi$ of the universal enveloping algebra $U(\pi^n\frg)$ of the $R$-Lie algebra $\pi^n\frg$ for all $n$. Note that ${\rm Gr}(\hat{A}_{n,K})$ is isomorphic to the symmetric algebra of the $R/\pi R$-vector space $\frg/\pi \frg$. In particular, the global dimension of $\hU(\pi^{n}\frg)_K$ is in fact equal to $d$ as follows from \cite[Prop. 9.1]{AW} applied to the augmentation character $\hU(\pi^{n}\frg)_K\rightarrow K$ given by $x=0$ for all $x\in\pi^n\frg$. Since $F_0A=R$, the Arens-Michael envelope $\hat{A}_K$ equals the completion of $U(\frg_K)$ with respect to {\it all} submultiplicative seminorms on the abstract $K$-algebra $U(\frg_K)$. This completion was first introduced and studied in \cite{SchmidtSTAB} and \cite{SchmidtBGG}.
For future reference we restate its faithful flatness property.

\begin{thm}\label{thm-Ug}
The natural homomorphism $U(\frg_K)\rightarrow \hU(\frg_K)$ is faithfully flat.
\end{thm}

\vskip5pt 

As a second example we consider a smooth affine integral scheme $X$ of finite type over $R$ whose closed fibre is integral. We assume that the locally free module of differentials $\Omega_{X/R}$ is already free, say of rank $d$. Let $A:=\cD(X)$ be the ring of (crystalline) global differential operators on $X$ with its natural filtration.\footnote{Note that $\cD(X)$ coincides with the {\it derivation ring} of $\cO(X)$ as studied in \cite[15.1]{MCR}.} In particular, $F_0A=\cO(X)$, the ring of global sections of $X$. Then $A$ satisfies all our requirements: again, (1) follows by definition of the filtration and (2) follows from $F_0A=\cO(X)$ and our assumptions on $X$. It is well-known that the graded ring of $\cD(X)$ equals the symmetric algebra of the $\cO(X)$-module consisting of the global vector fields on $X$ whence (3).

\vskip5pt

More generally, the enveloping algebra of a Lie algebroid \cite{Rinehart} gives rise to many examples. Let us briefly recall the definition (taken from \cite{ArdakovICM}). Let $R\rightarrow S$ be a ring homomorphism to some commutative ring $S$. A {\it Lie algebroid} is a pair $(L,a)$ consisting of an $R$-Lie algebra and $S$-module $L$, together with an $S$-linear $R$-Lie algebra homomorphism $a$ from $L$ to the $R$-linear derivations of $S$, such that $[v,sw]=s[v,w]+a(v)(s)w$ for all $v,w\in L$ and $s\in S$. It is possible to form a unital associative $R$-algebra $U(L)$ called the {\it enveloping algebra} of $(L,a)$ which is generated as a $R$-algebra by $S$ and $L$ subject to appropriate natural relations. Whenever $L$ is a projective $S$-module, $U(L)$ has a natural positive filtration with associated graded ring the symmetric algebra $\Sym_S(L)$. Suppose now that $L$ is already a free $S$-module, say of rank $d$. Then $F_0A=S$ and $A:=U(L)$ satisfies all our requirements if and only if $F_0A=S$ satisfies (2).
Our two first examples above are the special cases $S:=R$ and $(L,a):=(\frg,0)$ respectively $S:=\cO(X)$ and $(L,a):=(\Omega_{X/R}^\vee(X),id)$.

\section{From $D(\frg,P)$-modules to $D(G)$-modules}

We consider the locally $L$-analytic groups $P$ and $G$ as well as the maximal compact subgroup $G_0\subseteq G$. We let $P_0= G_0 \cap P$.
The locally analytic distribution algebras with coefficients in $K$ are denoted by $D(P), D(G), D(P_0)$ and $D(G_0)$. In this section, we will consider a certain functor $\cF_P^G(.)'$ from Lie algebra representations of $\frg$ endowed with a compatible locally analytic action of $P$ to locally analytic $G$-representations. This functor, or rather its restriction to certain highest weight categories was introduced and studied in \cite{OrlikStrauchJH}. To alleviate notation, we denote the universal enveloping algebra of the base change to $K$ of the $L$-Lie algebra $\frg$ by $U(\frg)$.

\vskip5pt

The group $G$ and its subgroup $P$ act via the adjoint representation on the Lie algebra $\frg$. We denote by

$$D(\frg,P):=D(P)\otimes_{U(\frp)} U(\frg)$$ 

\vskip8pt  

the corresponding skew-product ring. Similarly, we denote by $D(\frg,P_0)$ the skew-product ring $D(P_0)\otimes_{U(\frp)} U(\frg)$.

\begin{lemma}\label{lem-skewring} The natural linear map $D(\frg,P)\ra D(G)$ is an injective ring homomorphism with image equal to the subring of $D(G)$ generated by $D(P)$ and $U(\frg)$.
\end{lemma}

\begin{proof}
Let $U^{-}_P$ be the group of points of the opposite unipotent radical of $P$ and let $\fru^{-}_P$ be its Lie-algebra. In particular,
$\frg=\frp\oplus\fru^{-}_P$.
The multiplication map $P\times U^{-}_P\ra G$ is injective and induces an injective homomorphism $D(P\times U^{-}_P)\ra D(G)$.
The linear map appearing in the lemma is injective being the composite of the injective linear maps

$$ D(\frg,P)=D(P)\otimes_K U(\fru^{-}_P)\lra D(P)\otimes_K D(U^{-}_P)\lra D(P\times U^{-}_P)\lra D(G) \;.$$ 

\vskip8pt 

The remaining assertions are clear.
\end{proof}

An obvious variant of the above proof for the group $G_0$ shows that the natural linear map $D(\frg,P_0)\ra D(G_0)$ is an injective ring homomorphism with image equal to the subring of $D(G_0)$ generated by $D(P_0)$ and $U(\frg)$.

\begin{lemma} One has 

$$D(G)=D(G_0)\otimes_{D(\frg,P_0)} D(\frg,P)$$ 

\vskip8pt  

as bimodules. In particular, 

$$D(G)\otimes_{D(\frg,P)} M= D(G_0)\otimes_{D(\frg,P_0)} M$$ 

\vskip8pt  

for any ${D(\frg,P)}$-module $M$.
\end{lemma}

\begin{proof}
The bimodule map equal to the composite

$$D(G_0)\otimes_{D(P_0)} D(P)\lra D(G_0)\otimes_{D(\frg,P_0)} D(\frg,P)\ra D(G)$$ 

\vskip8pt  

is an isomorphism according to \cite[Lem. 6.1]{ST05}. Since the first map is surjective, both individual maps are isomorphisms as well. The second statement is clear.
\end{proof}

We consider the functor 

$$M\mapsto \cF_P^G(M)':= D(G)\otimes_{D(\frg,P)} M$$ 

\vskip8pt  

from ${D(\frg,P)}$-modules to $D(G)$-modules. Here, we follow the notation of \cite{OrlikStrauchJH}, compare in particular Prop. 3.7 in loc.cit. If the parabolic subgroup $P$ is clear from the context, we will occasionally abbreviate 

$$\bM:=\cF_P^G(M)' \;.$$ 

\vskip8pt 

\begin{lemma}\label{coadmissible} If $M$ is finitely generated as $U(\frg)$-module, then $\bM$ is coadmissible.
\end{lemma}

\begin{proof}
As a $D(G_0)$-module we have $\bM= D(G_0)\otimes_{D(\frg,P_0)} M$ according to the preceding lemma. The group $P_0$ is topologically finitely generated. Let $p_1,...,p_r$ be a set of topological generators and let $m_1,...,m_s$ be a set of generators for the $U(\frg)$-module $M$. Since $U(\frg)$ is noetherian, the $D(G_0)$-module
$D(G_0)\otimes_{U(\frg)} M$ is finitely presented and hence coadmissible. Consider its submodule $N$ finitely generated by the elements $\delta_{p_i}\otimes m_j -
1\otimes \delta_{p_i}m_j$. Then $N$ is coadmissible and it suffices to see that $N$ equals the kernel of the natural surjection $D(G_0)\otimes_{U(\frg)} M\ra \bM$. To this end, observe that $M$ equals a countable union of vector subspaces of finite dimension. We give $M$ the finest locally convex topology and let $W:=D(G_0)\otimes_K M$ have its projective tensor product topology, e.g. \cite[\S17B]{NFA}. Then $W$ satisfies the assumptions of \cite[Prop. 8.8]{NFA} and so the surjective continuous linear map $W\ra D(G_0)\otimes_{U(\frg)} M$ is open. In other words, the canonical topology on the coadmissible module $D(G_0)\otimes_{U(\frg)} M$ equals the quotient topology of $W$ by a suitable closed subspace. Hence, if $\delta_n\rightarrow\delta$ is a convergent sequence in $D(P_0)$, then for any $m\in M$ 

$$(\delta_n\otimes m-1\otimes\delta_n m)\rightarrow (\delta\otimes m-1\otimes\delta m)$$ 

\vskip8pt  

is a convergent sequence in the coadmissible module $D(G_0)\otimes_{U(\frg)} M$. The abstract group ring $K[P_0]$ is dense in $D(P_0)$ according to \cite[Lem. 3.1]{ST02b} which implies the claim.
\end{proof}

We now start a more detailed analysis of the module $\bM$ closely following the discussion in \cite[5.5]{OrlikStrauchJH}.
We put

\vspace{-0.3cm}
$$\kappa = \left\{\begin{array}{lcl} 1 & , & p>2 \\
2 & , & p=2 \end{array} \right.$$ 

\vskip8pt 

Let in the following $r$ always denote a real number in $(0,1) \cap p^\bbQ$ with the property:

\begin{numequation}\label{r_and_s}
\mbox{there is $m \in \bbZ_{\ge 0}$ such that $s = r^{p^m}$ satisfies $\frac{1}{p} < s$ and $s^{\kappa} < p^{-1/(p-1)}$ .}
\end{numequation}

For such numbers $r$ we let $D_r(G_0)$ and $D_r(P_0)$ be the Banach algebras appearing in loc.cit. Let us briefly sketch their construction. One chooses suitable uniform pro-$p$ groups $H\subset G_0$ and $H^+:=H\cap P_0$ such that $H$ is open normal in $G_0$. The distribution algebras of $H$ and $H^+$ admit canonical $r$-norms coming from the canonical $p$-valuation on the group \cite{ST03}. The rings $D(G_0)$ resp. $D(P_0)$ are finite free ring extensions over $D(H)$ resp. $D(H^+)$ and carry the corresponding maximum norms. The rings $D_r(G_0)$ resp. $D_r(P_0)$ are the associated completions. They define the Fr\'echet-Stein structure of $D(G_0)$ resp. $D(P_0)$. Let $U_r(\frg)$ and $U_r(\frp)$ be the topological closure of $U(\frg)$ in $D_r(G_0)$ and $U(\frp)$ in $D_r(P_0)$ respectively. Put

$$D_r(\frg,P_0):= D_r(P_0)\otimes_{U_r(\frp)} U_r(\frg) \;.$$ 

\vskip8pt 

An argument completely analogous to Lem. \ref{lem-skewring} shows that the natural linear map $D_r(\frg,P_0)\ra D_r(G_0)$ is an injective ring homomorphism with image equal to the subring of $D_r(G_0)$ generated by $D_r(P_0)$ and $U_r(\frg)$. If $HP_0$ denotes the subgroup of $G_0$ generated by $H$ and $P_0$, the intersection 

$$P_{0,r}:=HP_0\cap D_r(\frg,P_0)$$ 

\vskip8pt  

is thus well-defined.

\begin{lemma}\label{P_{0,r}} The set $P_{0,r}$ is an open normal subgroup of $HP_0$. One has 

$$D_r(G_0) =  \bigoplus_{g \in G_0/P_{0,r}} \delta_g D_r(\frg,P_0) \;.$$ 

\vskip8pt 
\end{lemma}

\begin{proof}
This follows form \cite[5.6]{OrlikStrauchJH}.
\end{proof}

We let

$$M_r:= U_r(\frg) \otimes_{U(\frg)} M, \hskip20pt \bM_r := D_r(G_0) \otimes_{D(G_0)} \bM = D_r(G_0) \otimes_{D(\frg,P_0)} M \;.$$ 

\vskip8pt 

For $g\in G$ we denote by ${\rm Ad}(g)$ the automorphism of $U(\frg)$ (or $U_r(\frg)$) induced by the left conjugation action $h\mapsto ghg^{-1}$ of $g$ on $G$. We note that the group $P_0$ acts on $M_r$ via $p.(\frx\otimes m):=({\rm Ad}(p)(\frx)) \otimes pm$.

\begin{lemma} The natural map 

$$M_r\stackrel{\simeq}{\lra} D_r(\frg,P_0)\otimes_{D(\frg,P_0)} M$$ 

\vskip8pt  

induced from the map $U_r(\frg)\ra D_r(\frg,P_0), \frx\mapsto 1\otimes\frx$ is
bijective.
\end{lemma}

\begin{proof} The ring $D_r(P_0)$ is a finite and free module over $U_r(\frp)$ on a basis given by distributions $\delta_p$ with $p\in P_0$. The $(P_0,U_r(\frg))$-module structure on $M_r$ therefore extends to a module structure over the ring $D_r(\frg,P_0)$. The resulting map $D_r(\frg,P_0)\otimes_{D(\frg,P_0)} M\ra M_r$ provides an inverse for the map in question.
\end{proof}

\vskip8pt

Using the two lemmas we can derive the following decomposition of $\bM_r$ as $U_r(\frg)$-module,

\vspace{-0.3cm}
\begin{numequation}\label{decomp}
\bM_r = D_r(G_0) \otimes_{D(\frg,P_0)} M = D_r(G_0) \otimes_{D_r(\frg,P_0)} M_r = \bigoplus_{g \in G_0/P_{0,r}} \delta_g\star M_r,
\end{numequation}

where $\delta_g\star M_r$ denotes the twist of the $U_r(\frg)$-module $M_r$ with the automorphism ${\rm Ad}(g)$ in the sense of Lem. \ref{lemma-twisting}.

\begin{prop} On the abelian category of $D(\frg,P)$-modules which are finitely generated as $U(\frg)$-modules, the correspondence $M\mapsto\bM$ constitutes an exact and faithful functor with trivial kernel (i.e. $\bM=0$ implies $M=0$).
\end{prop}

\begin{proof} The projective limit $\hat{U}(\frg)=\varprojlim_r U_r(\frg)$ is Fr\'echet-Stein and equals the Arens-Michael envelope of $U(\frg)$. If $M\neq 0$, then 

$$\hat{M}:=\hat{U}(\frg)\otimes_{U(\frg)} M\neq 0$$ 

\vskip8pt  

according to Thm. \ref{thm-Ug}. Moreover, since $M$ is finitely generated, the $\hat{U}(\frg_K)$-module $\hat{M}$ is finitely presented and hence coadmissible. Thus, $\hat{M}=\varprojlim_r M_r$ which implies $M_r\neq 0$ for a cofinal family of values of $r$. According to (\ref{decomp}), $\bM_r\neq 0$ for these $r$. Since $\bM$ is coadmissible (\ref{coadmissible}), this implies $\bM\neq 0$. Moreover, a sequence of coadmissible $D(G_0)$-modules is exact if and only if this is true after base extension to $D_r(G_0)$ for a cofinal family of values of $r$. The decomposition (\ref{decomp}) is natural in $M$. Since the functor $M\mapsto M_r$ is exact \cite[Rem. 3.2]{ST03}, so is the functor $M\mapsto\bM$.
The faithfulness is now a formal consequence.
\end{proof}

\vskip5pt

We compute a class of examples related to {\it locally analytic parabolic induction}. Recall the Levi decompositions $P=L_P\cdot U_P$ and $\frp=\frl_P\oplus\fru_P$.
Let $V$ be a locally analytic $L_P$-representation on a finite-dimensional $K$-vector space. We set $\fru_PV=0$ and consider $V$ a $U(\frp)$-module.
The induced $U(\frg)$-module $M(V)=U(\frg)\otimes_{U(\frp)} V$ is then naturally a $D(\frg,P)$-module which is finitely generated over $U(\frg)$. Indeed, we have the diagonal action of $L_P$ on the tensor product $M(V)$ where $L_P$ acts on the factor $U(\frg)$ via the adjoint action. It extends extends to a $D(L_P)$-action and it suffices therefore to check that the $\fru_P$-action extends compatibly to $D(U_P)$. However, the action of the Lie algebra $\fru_{P}$ even integrates uniquely to an algebraic action of $U_P$ on $M(V)$ as follows.
Given an element $u = \exp(\frx)\in {\bf U_P}(\overline{K})$, where $\overline{K}$ denotes an algebraic closure of $K$, we define $\rho(u) := \sum_{n \ge 0} \frac{\rho(\frx)^n}{n!}$, where $\rho(\frx)^n = 0$ for $n \gg 0$. The representations of $L_P$ and $U_P$ are compatible in the sense that $h \circ \rho(u) \circ h^{-1} = \rho(\Ad(h)(u))$, for $h \in L_P$, $u \in U_P$. Hence, $M(V)$ is a $D(P)$-module and then even a $D(\frg,P)$-module as claimed.

\begin{prop}\label{prop-induced_modules}
The map $V\rightarrow M(V), v\mapsto 1\otimes v$ induces an isomorphism of $D(G)$-modules 

$$D(G)\otimes_{D(P)} V\car\cF_P^G(M(V))' \;.$$ 

\vskip8pt 
\end{prop}

\begin{proof}
The map 

$$\cF_P^G(M(V))'\rightarrow D(G)\otimes_{D(P)} V, \delta\otimes (x\otimes v)\mapsto (\delta x)\otimes v$$ 

\vskip8pt 

for $\delta\in D(G), x\in U(\frg), v\in V$ is well-defined and provides a two-sided inverse.
\end{proof}

Remark: The module $D(G)\otimes_{D(P)} V$ is dual to the locally analytic parabolic induction ${\rm Ind}_P^G(V')$.

\vskip5pt

In the following we will investigate the behavior of the functor $\cF_P^G(.)'$ in terms of dimensions. To this end, recall that the ring $U(\frg)$ is a noetherian Auslander regular ring of global dimension $d:={\rm dim}_L\;\frg$. For a finitely generated $U(\frg)$-module $M$ we therefore have its canonical dimension $\dim_{U(\frg)} M:=d-j_{U(\frg)}(M)$, cf. \ref{sec-dim}.

\vskip5pt

Remark: Traditionally, dimension theory over the ring $U(\frg)$ is developed using the so-called {\it Gelfand-Kirillov dimension}, cf. \cite{Jantzen_Einhuellende}. However, it follows from \cite[Remark 5.8 (3)]{Levasseur} together with \cite[Prop. 8.1.15 (iii)]{MCR} that for finitely generated $U(\frg)$-modules, Gelfand-Kirillov dimension coincides with canonical dimension.

\vskip5pt

On the other hand, for any compact open subgroup $H\subseteq G$ and a coadmissible $D(H)$-module $M$, we define 

$$\dim_{D(H)}M:=d-j_{D(H)}(M) \;.$$ 

\vskip8pt 

If $D(H)=\varprojlim_r D_r(H)$ is a Fr\'echet-Stein structure for $D(H)$ and $M_r:=D_r(H)\otimes_{D(H)} M$, then \begin{numequation}\label{equ-r}\dim_{D(H)}(M)=\sup_r \dim_{D_r(H)}(M_r)\end{numequation} according to \cite[\S8]{ST03}. Moreover, if $M$ is even a $D(G)$-module, then, according to \cite{ST03} and \cite{SchmidtAUS}, the number $\dim_{D(H)}M$ is independent of the choice of $H$. In this case, we denote it by $\dim_{D(G)}M$, or simply $\dim M$, if no confusion can arise, and call it the {\it canonical dimension} of the coadmissible $D(G)$-module $M$.

\vskip5pt

We shall also need the Arens-Michael envelope $\hat{U}(\frg)$ of $\frg$ as introduced in the preceding section. Recall that this is a Fr\'echet-Stein algebra equal to the completion of $U(\frg)$ with respect to all submultiplicative seminorms on $U(\frg)$. As such, it comes with a natural completion homomorphism
$U(\frg)\rightarrow \hU(\frg)$ which is faithfully flat, cf. Thm. \ref{thm-Ug}.

\begin{thm}\label{thm-functor_dim} If $M$ is a $D(\frg,P)$-module which is finitely generated as $U(\frg)$-module, then 

$$\dG_{D(G)} \bM = \dg M \;.$$ 

\vskip8pt 
\end{thm}

\begin{proof}
 It suffices to prove $\jGn(\bM)=\jU (M)$. The left-hand side of this identity equals $\min_r \jGnr(\bM_r)$ according to (\ref{equ-r}). Now $D_r(G_0)$ is a finite free $U_r(\frg)$-module on a basis which consists of units satisfying the assumptions of \cite[Lem. 8.8]{ST03}. Hence,
$\jGnr(\bM_r)=\jUr(\bM_r)$ for all $r$. By (\ref{decomp}) together with Lem. \ref{lemma-twisting}, we have 

$$\jUr(\bM_r)=\max_{g\in  G_0/P_{0,r}} \jUr(\delta_g\star M_r)=\jUr M_r \;.$$ 

\vskip8pt  

So it remains to show
$\jU (M)=\min_r \jUr(M_r)$. Since $\hat{M}:= \hat{U}(\frg)\otimes_{U(\frg)}M$ is coadmissible, we have 

$$\jU(M)=j_{\hU(\frg)}(\hat{M})=\min_r \jUr(M_r)$$ 

\vskip8pt  

according to (\ref{cor-basechangegrade}) and (\ref{cor-mingrade}).
\end{proof}

Combining the theorem with \cite[Lem. 8.9]{Jantzen_Einhuellende} gives the dimension of parabolically induced representations.

\begin{cor}\label{cor-dimVERMA}
One has $\dim \cF_P^G(M(V))' = \dim_L (\frg/\frp)$ where $\dim_L$ denotes vector space dimension.
\end{cor}

\section{Highest weight modules and dimension}

In this section we explain the relation to the parabolic BGG-categories for the pair $\frp\subseteq\frg$ appearing in \cite{OrlikStrauchJH} and compute the dimensions of certain irreducible $G$-representations occurring in the image of the functor $\cF_P^G$. As in the previous section, we make the general convention that, when dealing with universal enveloping algebras, we write $U(\frg)$, $U(\frp)$ etc. to denote the corresponding universal enveloping algebras {\it after base change to $K$}, i.e., what is precisely $U(\frg_K)$, $U(\frp_K)$ and so on.

\begin{para} {\it The category $\cO$ and its parabolic variants $\cO^\frp$.}
The category $\cO$ in the sense of Bernstein, Gelfand, Gelfand, cf. \cite{BGG2}, \cite{H1}, is defined for complex semi-simple Lie algebras. Here we consider the following variant for split reductive Lie algebras over a field of characteristic zero. Thus we let $\cO$ be the full subcategory of all $U(\frg)$-modules $M$ which satisfy the following properties:

\begin{enumerate}
\item $M$ is finitely generated as a $U(\frg)$-module. \smallskip
\item $M$ decomposes as a direct sum of one-dimensional $\frt_K$-representations. \smallskip
\item The action of $\frb_K$ on $M$ is locally finite, i.e. for every $m \in M$, the subspace $U(\frb)\cdot m \subset M$ is finite-dimensional over $K$.
\end{enumerate}

\noindent As in the classical case one shows that $\cO$ is a $K$-linear, abelian, noetherian, artinian category which is closed under submodules and quotients, cf. \cite[1.1, 1.11]{H1}. In particular, every object of $\cO$ has a Jordan-H\"older series and a simple object of $\cO$ is simple as abstract $U(\frg)$-module.


Following \cite{OrlikStrauchJH} we define a certain 'algebraic' subcategory of $\cO$. Note that by property (2), we may write any object $M$ in $\cO$ as a direct sum
\begin{equation}\label{eigenspace}
M=\bigoplus_{\lambda \in \frt^\ast_K} M_\lambda
\end{equation}
where $M_\lambda=\{m\in M \mid \forall \frx \in \frt_K: \frx \cdot m = \lambda(\frx) m  \}$ is the $\lambda$-eigenspace attached to $\lambda \in \frt^\ast_K =  \Hom_K(\frt_K,K)$. Let $X^\ast(\bT) = \Hom(\bT,\bbG_m)$ be the group of characters of the torus $\bT$ which we consider via the derivative as a subgroup of $\frt^\ast_K$.
\end{para}
We denote by $\cO_\alg$ the full subcategory of $\cO$
whose consisting of objects $M\in \cO$ where the $\frt_K$-module structure on every $M_\lambda$ lifts to an algebraic action of $\bT$.
Again, $\Oa$ is an abelian noetherian, artinian category which is closed under submodules and quotients.
The Jordan-H\"older series of a given $U(\frg)$-module lying in $\Oa$ is the same as the one considered in the category $\cO.$
\begin{exam}
For $\lambda\in \frt^\ast_K$, let $K_\lambda=K$ be the 1-dimensional $\frt_K$-module where the action is given by $\lambda$. Then
$K_\lambda$ extends uniquely to a $\frb_K$-module. Let

$$M(\lambda)=U(\frg) \otimes_{U(\frb)} K_\lambda \in \cO$$ 

\vskip8pt 

be the corresponding Verma module. Denote by $L(\lambda)\in \cO$  its simple quotient. Suppose the character $\lambda$ integrates to a locally analytic character of $T$. As we have explained before Prop.\ref{prop-induced_modules}, the module $M(\lambda)$ is then a $D(\frg,B)$-module finitely generated over $U(\frg)$ and the same holds true for $L(\lambda)$. In this situation, $M(\lambda)$ resp. $L(\lambda)$ is an object of $\Oa$ if and only if $\lambda \in X^\ast({\bf T}).$
\end{exam}

We shall also need the parabolic versions of the above categories. We define $\cO^\frp$ to be the category of $U(\frg)$-modules $M$ satisfying the following properties:

\begin{enumerate}
\item $M$ is finitely generated as a $U(\frg)$-module.
\item Viewed as a $\frl_{P,K}$-module, $M$ is the direct sum of finite-dimensional simple modules.
\item The action of $\fru_{P,K}$ on $M$ is locally finite.
\end{enumerate}

This is analogous to the definition over an algebraically closed field, cf. \cite[ch. 9]{H1}. Clearly, the category $\cO^\frp$ is a full subcategory of $\cO$. Furthermore, it is $K$-linear, abelian and closed under submodules and quotients, cf. \cite[9.3]{H1}. Hence the Jordan-H\"older series of every $U(\frg)$-module in $\cO^\frp\subset \cO$ lies in $\cO^\frp$ as well. If $Q$ is a standard parabolic subgroup with $Q\supset P$, then  $\cO^\frq \subset \cO^\frp.$ Finally, consider the extreme case $\frp=\frg$: the category $\cO^\frg$ consists of all finite-dimensional semi-simple $\frg_K$-modules. On the other hand, $\cO^\frb = \cO$.

\vskip8pt

Similarly as before we define a subcategory $\cO^\frp_\alg$ of $\cO^\frp$ as follows. Let ${\rm Irr}(\frl_{P,K})^{\rm fd}$ be the set of isomorphism classes of finite-dimensional irreducible $\frl_{P,K}$-modules.
Again, any object in $\cO^\frp$ has by property (2) a decomposition into $\frl_{P,K}$-modules

\begin{numequation}\label{isotypical_decom}
M= \bigoplus_{\fra \in {\rm Irr}(\frl_{P,K})^{\rm fd}} M_{\fra}
\end{numequation}

where $M_{\fra} \sub M$ is the $\fra$-isotypic part of the representation $\fra$. We denote by $\cO^\frp_\alg$ the full subcategory of $\cO^\frp$
consisting of objects $M$ of $\cO^\frp$ with the following property: if $M_\fra \neq 0$ (with the notation as in \ref{isotypical_decom}), then $\fra$ is the Lie algebra representation induced by a finite-dimensional algebraic $\bL_{\bP,K}$-representation, where $\bL_{\bP,K} = \bL_\bP \times_{\Spec(L)} \Spec(K)$.
Again, the category $\cO^\frg_\alg$ is contained in $\Oa$ and contains all finite-dimensional $\frg_K$-modules which are induced by algebraic ${\bf G}$-modules. Every object in $\cO^\frp_\alg$ has a Jordan-H\"older series which coincides with the Jordan-H\"older series in $\Oa$. If $M$ is an object of $\cO^\frp$, then $M$ is in $\cO^\frp_\alg$ if and only it is in $\cO_\alg$, cf. \cite[Lem. 2.8]{OrlikStrauchJH}.

\setcounter{enumi}{0}

\begin{exam}\label{Example_Verma}
Let $\Delta$ be the set of simple roots of ${\bf G}$ with respect to ${\bf T\subset B}$.
Let $\lambda \in\frt^*_K$ and set $I = \{\alpha \in \Delta \midc \langle \lambda, \alpha^\vee \rangle \in \bbZ_{\ge 0} \}.$
We let ${\bf P}={\bf P}_I$ is the standard parabolic subgroup of ${\bf G}$ attached to $I.$  Then $\lambda$ is dominant with respect to the reductive Lie algebra $\frl_P$. Denote by $V_I(\lambda)$ the corresponding irreducible finite-dimensional $\frl_{\frp}$-representation and consider the generalized Verma module (in the sense of Lepowsky \cite{Lepowsky})

$$M_I(\lambda)=U(\frg) \otimes_{U(\frp_I)} V_I(\lambda) \;.$$ 

\vskip8pt 

There is a surjective map 

$$M(\lambda) \rightarrow M_I(\lambda) \;,$$ 

\vskip8pt 

where the kernel is given by the image of $\oplus_{\alpha \in I} M(s_\alpha\cdot \lambda) \rightarrow M(\lambda)$. Now suppose the $\frl_{\frp}$-representation on $V_I(\lambda)$ integrates to a locally analytic $L_P$-representation. As we have explained before Prop.\ref{prop-induced_modules}, the module $M_I(\lambda)$ is then a $D(\frg,P)$-module and finitely generated over $U(\frg)$.
In this situation, $M_I(\lambda)$ is an object of $\cO^\frp_\alg$ if and only if the $\frl_{\frp}$-action on $V_I(\lambda)$ integrates to an algebraic $L_P$-action. This happens if and only if $\lambda\in X({\bf T})$. In this case, $L(\lambda)$ is an object of $\cO^\frp_\alg$, as well, cf. \cite[sec. 9.4]{H1}.
\end{exam}

Let $M$ be an object of $\cO^\frp_\alg$ as above. Then $M$ is the union of finite-dimensional $\frp_K$-modules. Denote by $X$ one of these finite-dimensional submodules. Then $X$ lifts uniquely to an algebraic $\bP_K$-representation \cite[Cor. 3.6]{OrlikStrauchJH}. Let us sketch the argument. The $U(\frp)$-module $X$, considered as a $U(\frl_\frp)$-module, decomposes into a
direct sum of isotypic modules $X_\fra$ and each module $X_\fra$ lifts uniquely to an algebraic representation of
$\bL_{\bP,K}$. The action of the Lie algebra $\fru_{\frp,K}$ integrates uniquely to an algebraic action of $\bU_\bP$ on $X$ in the manner we have explained before Prop. \ref{prop-induced_modules}. This shows that $X$ is uniquely endowed with an algebraic representation of $\bP_K$. Consequently, there is a unique $D(\frg,P)$-module structure on $M$ that extends its $U(\frg)$-module structure and such that the action of $U(\frp)$, as a subring of $U(\frg)$, coincides with the action of $U(\frp)$ as a subring of $D(P)$. Moreover, any morphism $M_1 \ra M_2$ in $\cO^\frp_\alg$ is automatically a homomorphism of $D(\frg,P)$-modules. In other words, we have a fully faithful embedding of categories

$$ \cO^\frp_\alg\hskip10pt\hookrightarrow\hskip10pt  {\rm category~of~all~}D(\frg,P){\rm -modules,~finitely~generated~over}~U(\frg) \;.$$ 

\vskip8pt 

We now explain how one may compute the dimensions of the irreducible $G$-representations that occur in the image of $\cO^{\frp}_{\rm alg}$ via the functor $\cF_P^G$.

\begin{dfn}\label{maximal} Let $M$ be an object of the category $\cO$. We call a standard parabolic subalgebra $\frp$ {\it maximal for $M$} if $M \in \cO^\frp$ and if $M \notin \cO^\frq$ for all parabolic subalgebras $\frq$ strictly containing $\frp$.
\end{dfn}

It follows from \cite[sec. 9.4]{H2} that for every object $M$ of $\cO$ there is unique standard parabolic subalgebra $\frp$ which is maximal for $M$.
The same definition applies for objects in the subcategory $\cO_{\alg}$ in which case we also say that the standard parabolic subgroup $P$ corresponding to $\frp$ is {\it maximal for $M$}. In that case $M$ lies in $\cO^\frp_\alg$. We recall \cite[Thm.5.3]{OrlikStrauchJH}.

\begin{thm}\label{irredG_0}
If the root system $\Phi = \Phi(\frg,\frt)$ has irreducible components of type $B$, $C$ or $F_4$, we assume $p > 2$, and if $\Phi$ has irreducible components of type $G_2$, we assume that $p > 3$. Let $M \in \cO_\alg^\frp$ be simple and assume that $P$ is maximal for $M$. Then
$\cF_P^G(M)'$ is a simple $D(G_0)$-module (and so, in particular, a simple $D(G)$-module).
\end{thm}

\vskip5pt

We let $\lambda\in\frt^*_K$ be the differential of a locally analytic character of $T$ and let $P=P_I$ be adapted to $\lambda$ in the sense of Example \ref{Example_Verma}. Then $\frp_K$ is maximal for $M(\lambda)$. Consider the simple quotient $L(\lambda)$ and the coadmissible $D(G)$-module

$$\bL_I(\lambda):=\cF_P^G(L(\lambda))' \;.$$ 

\vskip8pt 

We have $\dim \bL(\lambda)=\dim L(\lambda)$ by Thm. \ref{thm-functor_dim}. Moreover, if $\lambda\in X^*(\bT)$, then $L(\lambda)$ is an object of $\cO_\alg^\frp$ and the $D(G)$-module $\bL(\lambda)$ is simple under the assumptions of the preceding theorem. We therefore briefly recall the classical relation of $\dim L(\lambda)$ to classical Goldie rank polynomials. To this end, we need to introduce some extra notation following \cite[2.7]{Jantzen_Einhuellende}. We let $X^*(\bT)\subseteq \Lambda$ be the integral weight lattice and let $\Lambda^+$ and $\Lambda^{++}$ be the subsets of dominant resp. strictly dominant weights. For simplicity we assume $\lambda\in\Lambda$. Recall that the isomorphism classes of the $L(\mu), \mu\in\frt^*$ as well as the isomorphism classes of the $M(\mu), \mu\in\frt^*$ form two different $\bbZ$-bases of the Grothendieck group of the abelian category $\cO$, cf. \cite[4.5]{Jantzen_Einhuellende}. In particular, for any $\mu\in\frt^*$

$$[L(\mu)]=\sum_{\mu'\in\frt^*} (L(\mu):M(\mu'))[M(\mu')]$$ 

\vskip8pt  

for some uniquely determined coefficients $(L(\mu):M(\mu'))\in\bbZ$. For any $\mu\in\Lambda^{++}$, the number

$$ a_{\Lambda}(w,w'):=(L(w\cdot\mu):M(w'\cdot \mu))$$ 

\vskip8pt 

for $w,w'\in W$ is independent of the choice of $\mu$, cf. \cite[4.14]{Jantzen_Einhuellende}. We fix once and for all an element $t\in\frt$ such that $\alpha(t)=1$ for all $\alpha\in \Delta$. For fixed $w\in W$ we let $m=m_w\in\bbN_{\geq 0}$ be minimal such that

$$\tilde{f}^{\Lambda}_w:=\frac {1}{m!}\sum_{w'\in W}a_{\Lambda}(w,w')(w'^{-1}(t))^m\in  \Sym^m(\frt)$$ 

\vskip8pt 

is nonzero, cf. \cite[9.13]{Jantzen_Einhuellende}. The number $m_w$ does not depend on the particular choice of $t\in\frt^*$. In fact, different choices of $t$ lead to polynomials that differ by a scalar in $L^\times$, cf. \cite[14.7]{Jantzen_Einhuellende}.
The polynomial $\tilde{f}^{\Lambda}_w$ is, up to scaling, the so-called {\it Goldie rank polynomial} of $w\in W$.\footnote{The polynomials $\tilde{f}^{\Lambda}_w$ and their generalizations to arbitrary cosets $\frt^*/\Lambda$ were introduced and studied by Joseph and build a bridge between primitive ideals of $U(\frg)$, nilpotent adjoint orbits and the representation theory of $W$. For more details we refer to \cite[Kap. 14]{Jantzen_Einhuellende}.}

\vskip5pt

We pick $\mu\in\Lambda^+$ such that $\lambda=w\cdot\mu=w(\mu+\rho)-\rho$ for some $w\in W$, put

\begin{numequation}\label{equ-S}S:=B_\mu^0:=\{\alpha \in \Delta \midc \langle \mu+\rho, \alpha^\vee \rangle=0 \}\end{numequation}

and let $W_S$ be the subgroup of $W$ generated by all $s_\alpha, \alpha\in S$. Hence, $W_S$ coincides with the stabilizer $\{w'\in W: w'\cdot\mu=\mu\}$ according to \cite[2.5]{Jantzen_Einhuellende}. Let $W^S$ be the unique system of representatives of maximal length for the left cosets in $W/W_S$. Since $W\cdot\mu=W^S\cdot \mu$ we may and will assume that $w\in W^S$.

\begin{thm}\label{thm-dim_irred}
The module $L(\lambda)=L(w\cdot\mu)$ has the dimension
$${\rm dim}_{U(\frg)} L(w\cdot \mu)=\# \Phi^+ - m_w$$ \vskip8pt 
where $m_w$ denotes the degree of the polynomial $\tilde{f}_w^{\Lambda}$.
\end{thm}

\begin{proof}
Since $w\in W^S$, we have $B^0_\mu=S\subset \tau_{\Lambda}(w)$ according to \cite[2.7(1)]{Jantzen_Einhuellende} (and in the notation of loc.cit.).
We may therefore apply \cite[Satz 9.12]{Jantzen_Einhuellende} to obtain $\dim_{U(\frg)} L(w\cdot\mu)=\# \Phi^+ - m_w$. Note that all argument extend from the split semisimple case of loc.cit. to the more general split reductive case considered here and that Gelfand-Kirillov dimension may be replaced by canonical dimension.
\end{proof}

Remark: The wish to explicitly compute the polynomial $\tilde{f}_w^{\Lambda}$ and its degree led to the formulation of the so-called Kazhdan-Lusztig conjecture \cite{KL_CoxeterHecke}. This conjecture is now a theorem thanks to the work of Beilinson-Bernstein \cite{BB81} and Brylinski-Kashiwara
\cite{BK81}.

\section{Application to equivariant line bundles on Drinfeld's upper half space}
In this section we explain briefly how the results of the preceding sections combined with a theorem from \cite{OrlikStrauchJH} allow to compute the dimension of representations coming from line bundles on Drinfeld's half space.

\vskip5pt

We let $\bG =\GL_{d+1}$. Moreover, $\bB\subset\GL_{d+1}$ equals the Borel subgroup of lower triangular matrices and $\bT\subset\bB$ the diagonal torus. For a decomposition $(n_1,...,n_s)$ of $d+1$ the symbol $\bP_{n_1,...,n_s}$ denotes the corresponding lower standard parabolic subgroup of $\GL_{d+1}$ with Levi subgroup $\bL_{n_1,...,n_s}.$

\vskip5pt

Let $\cX$ be Drinfeld's half space of dimension $d\geq1$ over $K$. This is a rigid-analytic variety over $K$ given by the complement of all  $K$-rational hyperplanes in projective space $\bbP_K^d,$ i.e.,

$$\cX = \bbP_K^d\setminus \bigcup \nolimits_{H \subsetneqq K^{d+1}}\mathbb \bbP(H) \;,$$ 

\vskip8pt 

where $H$ runs over the set of $K$-rational hyperplanes in $K^{d+1}.$ There is  natural action of $G = \GL_{d+1}(K)$ on $\cX$ induced by the algebraic
action $m: \bG \times \bbP^d_K \rightarrow \bbP^d_K$ of $\bG$ defined by

$$g\cdot [q_0:\cdots :q_d]:=m(g,[q_0:\cdots :q_d]):= [q_0:\cdots :q_d]g^{-1} \,.$$ 

\vskip8pt 

\noindent Let $s\in\bbZ$ and denote by $\lambda'=(s,...,s)\in\bbZ^d$ the constant integral weight for $\GL_d$. Let $r=\lambda_0\in\bbZ$ and set

$$\lambda=(r,s,...,s)\in\bbZ^{d+1} \;.$$ 

\vskip8pt  

We denote by $\cL_\lambda$ the homogeneous line bundle on $\bbP_K^d=\GL_{d+1}/\bP_{1,d}$ such that its fibre in the base point is the irreducible algebraic $\bL_{1,d}$-representation corresponding to $\lambda$. Then we obtain $\cL_\lambda = \cO(r-s)$ where the $\bG$-linearization is given by the tensor product of the natural one on $\cO(r-s)$ with $\det^s$. The space of global sections $H^0(\cX,\cL_\lambda)$ is a coadmissible $D(G)$-module. We may compute its dimension as follows.

\vskip5pt

Put $w_j:=s_j\cdots s_1,$ where $s_i\in W$ is the (standard) simple reflection in the Weyl group $W \cong S_{d+1}$ of $G$.
Recall that $\cdot$ denotes the dot action of $W$ on $X^\ast({\bf T})_{\bbQ}$.
There is at most one integer $0\leq i_0 \leq d$, such that $H^{i_0}(\bbP^d_K, \cL_\lambda)$ is non-vanishing
which is $i_0=0$ for $r\geq s$ resp. $i_0=d$ for $s \geq r+d+1$. Otherwise, there is a unique integer $i_0 < d$ with $w_{i_0}\cdot \lambda = w_{i_0+1}\cdot \lambda.$ This is the case for $0\leq i_0=s-r-1<d+1.$ We put

$$\mu_{i,\lambda}:=\left\{
\begin{array}{cc}
w_{i-1}\cdot \lambda  & : i \leq i_0 \\
w_{i} \cdot \lambda & : i > i_0 \,
\end{array} \right\} \,,i=1,\ldots,d \;.$$ 

\vskip8pt

This is a ${\bf L_{(i,d-i+1)} }$-dominant weight with respect to the Borel subgroup ${\bf L_{(i,d-i+1)}} \cap {\bf B^+}$ where $\bB^+$ denotes the upper triangular matrices in $\GL_{d+1}$. Consider the block matrix

$$ z_j:=\left( \begin{array}{cc} 0  & I_j \\  I_{d+1-j} & 0 \end{array} \right) \in G \;,$$ 

\vskip8pt 

where $I_k\in {\rm GL}_{k}(K)$ denotes the $k\times k$-identity matrix.
We may regard $z_j$ as an element of $W$ and consider the weights $z_j^{-1}\cdot \mu_{j,\lambda}$ for any $j=0,...,d-1$. For each $j$ we choose an element $v_j\in W$ such that $$v_j^{-1}\cdot (z_j^{-1}\cdot\mu_{j,\lambda})\in\Lambda^+.$$ \vskip8pt  As explained in the discussion following (\ref{equ-S}), we can and will here assume additionally that $v_j$ lies in the subset $W^{S_j}\subset W$ corresponding to
$$S_j:=B^0_{v_j^{-1}\cdot (z_j^{-1}\cdot\mu_{j,\lambda})}.$$ \vskip8pt 

\begin{thm}\label{TheoremVB} The coadmissible $D(G)$-module $H^0(\cX,\cL_\lambda)$ has the dimension

$$\dim H^0(\cX,\cL_\lambda)=\#\Phi^+-\min_{j=0,...,d-1} m_{v_j} \;.$$ 

\vskip8pt 
 
Here, $m_{v_j}$ denotes the degree of the polynomial $\tilde{f}^\Lambda_{v_j}$.
\end{thm}

\begin{proof}
We abbreviate $\cL(\cX):=H^0(\cX,\cL_\lambda)$. For $j=0,...,d-1$ the $U(\frg)$-module $L(z_j^{-1}\cdot\mu_{j,\lambda})$ lies in the category  $\cO^{\frp_{j+1,d-j}}_\alg$. The parabolic $\frp_{j+1,d-j}$ is maximal for it, cf. \cite[Prop. 7.5]{OrlikStrauchJH}, and so under the assumption of Thm. \ref{irredG_0} the $D(G)$-module $\cF^G_{P_{(j+1,d-j)}}( L(z_j^{-1}\cdot\mu_{j,\lambda}))$ is simple, but we do not need this. In any case, there is a filtration of $\cL(\cX)$ by coadmissible $D(G)$-submodules

\begin{numequation}\label{filtration}
 \cL(\cX)=\cL(\cX)^0 \supset \cL(\cX)^{1} \supset \cdots \supset \cL(\cX)^{d-1} \supset \cL(\cX)^d =  H^0(\bbP^d,\cL),
\end{numequation}
\noindent such that for $j=0,\ldots,d-1,$ there are $D(G)$-module extensions

\vspace{-0.3cm}
\begin{numequation}\label{extensions}
 0 \rightarrow \cF^G_{P_{(j+1,d-j)}}( L(z_j^{-1}\cdot\mu_{j,\lambda}))\rightarrow (\cL(\cX)^{j}/\cL(X)^{j+1}) \rightarrow \cZ_j\rightarrow 0
\end{numequation}

where $\cZ_j$ is dual to a locally algebraic $G$-representation, cf. \cite[Cor. 7.6]{OrlikStrauchJH} and \cite[Cor. 2.2.9]{Orlik08}.
Since $\cZ_j$ has dimension zero, we obtain

$$\dim (\cL(\cX))=
\max_{j=0,...,d-1} \dim \cF^G_{P_{(j+1,d-j)}}( L(z_j^{-1}\cdot\mu_{j,\lambda}))$$ 

\vskip8pt  

and so Thm. \ref{thm-dim_irred} yields the assertion.
\end{proof}

\section{Dimensions in the unitary principal series for ${\rm GL}_2(\bbQ_p)$}

Let $L=\bbQ_p$ in this section. We first discuss some preliminaries concerning locally analytic vectors in Banach space representations.
In the second part we apply this to the unitary principal series for ${\rm GL}_2(\bbQ_p)$ as introduced by Colmez \cite{Colmez_trianguline}.

\vskip5pt

We consider an arbitrary locally $\bbQ_p$-analytic group $G$ of dimension $d$. Fix some arbitrary compact open subgroup $H\subset G$ and let $o_K[[H]]$ be the completed group ring of $H$ with coefficients in $o_K$ and put $K[[H]]:=K\otimes o_K[[H]].$ Both rings are noetherian. The abelian category of admissible Banach space representations of $G$ is anti-equivalent to the category of coadmissible modules, that is, finitely generated $K[[H]]$-modules equipped with a compatible linear action of $G$ \cite[Thm. 3.5]{ST02a}. If $H'\subseteq H$ is another compact open subgroup, then the natural map $K[[H']]\rightarrow K[[H]]$ is a finite (free) ring extension. The category of coadmissible modules is therefore independent of the choice of $H$. Applying \cite[Lem. 8.8]{ST03} to this latter ring extension shows $j_{K[[H']]}(M)=j_{K[[H]]}(M)$ and one may therefore unambiguously define the dimension of a coadmissible module $M$ to be

$$\dim M:= d-j_{K[[H]]}(M) \;.$$ 

\vskip8pt 

Remark: If $H$ has no element of order $p$, then $K[[H]]$ is an Auslander regular ring of global dimension $d$ \cite[Thm. 3.29]{Venjakob02}.

\vskip5pt

The passage to locally analytic vectors $V\mapsto V_{an}$ is an exact functor from admissible Banach space representations to admissible locally analytic representations \cite[ \S7]{ST03}. Let $V'$ and $(V_{an})'$ be the corresponding coadmissible modules.

\begin{prop}\label{prop-anvect}
One has $\dim V'=\dim (V_{an})'$ for any admissible Banach space representation $V$ of $G$.
\end{prop}

\begin{proof}
As representations of $H$, the functor $V\mapsto V_{an}$ is given on the level of coadmissible modules as the base change $M\mapsto D(H)\otimes_{K[[H]]} M$. Since the homomorphism $K[[H]]\rightarrow D(H)$ is faithfully flat \cite[Thm. 5.2]{ST03}, the assertion follows from Lem. \ref{lemma-basechangegrade}.
\end{proof}

We now specialize to certain reductive groups $G$. To this end, let $\bG$ be a connected split reductive group scheme over $o_L$.
Let $\bar{\kappa}$ be an algebraic closure of $\kappa$, the residue field of $o_L$. Let us consider the following three hypothesis on the geometric closed
fibre $\bG_{\bar{s}}=\bG\otimes_{o_L}\bar{\kappa}$ of $\bG$ which
are familiar from the theory of modular Lie algebras (cf.
\cite{JantzenCharp}, 6.3).

\begin{itemize}
    \item[(H1)] The derived group of $\bG_{\bar{s}}$ is (semisimple) simply
    connected.
    \item[(H2)] The prime $p$ is good for the $\bar{\kappa}$-Lie algebra $Lie(\bG_{\bar{s}})$.
    \item[(H3)] There exists a $\bG_{\bar{s}}$-invariant non-degenerate bilinear form on $Lie(\bG_{\bar{s}})$.
\end{itemize}

For example, the general linear group ${\rm GL}_n$ satisfies these conditions for all primes $p$ (using the trace form for {\rm (H3)}). Any almost simple and simply connected $\bG_{\bar{s}}$ satisfies
these conditions if $p\geq 7$ (and if $p$ does not divide $n+1$ in
case $\bG_{\bar{s}}$ is of type $A_n$). For a more detailed
discussion of these conditions we refer to loc.cit.

\vskip5pt

We assume from now on that (H1)-(H3) hold. As before, $\Phi^+$ denotes a set of positive roots of $\bG$ and the number $r$ denotes half the dimension of the minimal nilpotent coadjoint orbit of $\bG_{\bar{s}}$
(\cite{CollingwoodMcGovern}, Rem. 4.3.4). We let $G$ from now be on a locally $L$-analytic group whose $L$-Lie algebra $Lie(G)$ is isomorphic to
$\frg_L:=L\otimes_{o_L}\frg$. Let us identify these Lie algebras
via such an isomorphism. Let $d={\rm dim}_L\; \frg_L$.

\begin{prop}\label{prop-bounds}
Let $V$ be an admissible $G$-Banach space representation. If $V^{an}$ has an infinitesimal character, then ${\rm dim}(V)\leq 2\cdot \#\Phi^+$.
If ${\rm dim}(V)\geq 1$, then ${\rm dim}(V)\geq r$.
\end{prop}

\begin{proof}
The second statement follows from the main result of \cite{AW}, extended to reductive groups satisfying (H1)-(H3) in \cite{SchmidtDIM}.
Suppose now that $V^{an}$ has an infinitesimal character. By \cite[9.4/9.6]{SchmidtDIM} it suffices to show that, for any $n\geq 1$,
the dimension of a finitely generated module $M$ over the Auslander regular ring $\widehat{U(\frg)}_{n,K}$ with infinitesimal character is bounded above by $2\cdot \#\Phi^+$. Here, $\widehat{U(\frg)}_{n,K}$ denotes the $\pi$-adic completion (with subsequent inversion of $\pi$) of the universal enveloping algebra $U(\pi^n\frg)$ for a choice of uniformizer $\pi$ of $o_L$. We may choose a good double filtration for $M$ and form its double graded module ${\rm Gr}(M)$ in the sense of \cite[3.2]{AW}. The latter is a finitely generated module over ${\rm Gr}(\widehat{U(\frg)}_{n,K})=\Sym (\frg_\kappa)$ whose support has dimension equal to the dimension of $M$. Since $M$ has a central character, ${\rm Gr}(M)$ is annihilated by $\Sym(\frg_\kappa)_+^{\bG_k}$, the ideal of invariant polynomials without constant term. Its support lies therefore in the cone of nilpotent elements of $\frg_\kappa$ which has dimension $2\cdot\#\Phi^+$.
\end{proof}

We recall that if $V$ is absolutely irreducible, then $V^{an}$ admits an infinitesimal character \cite{DospSchraen}.

\vskip16pt

We now let $\bG={\rm GL}_2$ and turn to the unitary principal series of $G={\rm GL}_2(\bbQ_p)$. As usual, $B\subset G$ denotes the Borel subgroup of upper triangular matrices. We fix a finite extension $\bbQ_p\subseteq K$ as a coefficient field for the representations. We denote the continuous character $x\mapsto x|x|$ of $\bbQ_p^\times$ by $\chi$.
Finally, $\cG_{\bbQ_p}=Gal(\bar{\bbQ}_p/\bbQ_p)$ denotes the absolute Galois group of $\bbQ_p$.

In \cite{Colmez10} Colmez establishes a correspondence 

$$V\mapsto\Pi(V)$$ 

\vskip8pt  

from absolutely irreducible $2$-dimensional representations $V$ of $\cG_{\bbQ_p}$ over $K$ to absolutely irreducible unitary admissible $G$-representations. This correspondence is based on the construction of a $G$-representation $D(V) \boxtimes \bbP^1$ attached to $V$ with central character $\delta(x)=\chi^{-1} \det_V(x)$ where $\det_V$ is the character of $\bbQ_p^\times$ corresponding by local class field to the determinant of $V$. The representation
$D(V)\boxtimes \bbP^1$ is an extension of $\Pi(V)$ by its dual twisted by $\delta\circ\det$. In particular, the central character of $\Pi(V)$ equals $\delta$ and Prop. \ref{prop-bounds} implies $\dim \Pi(V)\leq 2.$ In the remainder of this section, we will determine $\dim \Pi(V)$ in case $\Pi(V)$ belongs to the unitary principal series, i.e. in case $V$ is trianguline \cite{Colmez_trianguline}.

\vskip5pt
In the following, all $(\varphi,\Gamma)$-modules are taken over the classical Robba ring $\sR$. Given a continuous character $\eta: \bbQ_p^\times \rightarrow K^\times$, the associated $(\varphi,\Gamma)$-module of rank $1$ is denoted by $\sR(\eta)$.
Recall that a $2$-dimensional Galois representation is called {\it trianguline}, if the associated \'etale $(\varphi,\Gamma)$-module is an extension of two (non-\'etales if $V$ is irreducible) modules of rank $1$. If $\eta_1,\eta_2$ are two continuous characters $\bbQ_p^\times\rightarrow K^\times$, we denote the locally analytic induction ${\rm Ind}^G_B(\eta_2\otimes\eta_1\chi^{-1})$ simply by $B^{an}(\eta_1,\eta_2)$ (note the reversed order of the $\eta_i$!).

\begin{prop}
One has $\dim\Pi(V)=1$ for any irreducible trianguline representation $V$.
\end{prop}

\begin{proof}
Let $\Delta(s)$ be the \'etale $(\varphi,\Gamma)$-module associated with $V$. Here, $s=(\delta_1,\delta_2,\sL)$ is the associated parameter consisting of continuous characters $\delta_1,\delta_2: \bbQ_p^\times\rightarrow K^\times$ and an element $\sL\in\bbP({\rm Ext}^1(\sR(\delta_1),\sR(\delta_2))).$ In \cite[Thm. 0.7]{Colmez_LAvec} (compare also \cite{LXZ_LAvec}) the locally analytic representation $\Pi(V)^{an}$ is computed. Either we have
the exact sequence of locally analytic $G$-representations

$$0\rightarrow B^{an}(\delta_1,\delta_2)\rightarrow \Pi(V)^{an}\rightarrow B^{an}(\delta_2,\delta_1)\rightarrow 0$$ 

\vskip8pt 

(the {\it generic} case) or we have an exact sequence of locally analytic $G$-representations

$$0\rightarrow E_{\sL}\rightarrow \Pi(V)^{an}\rightarrow B^{an}(\delta_2,\delta_1)\rightarrow 0$$ 

\vskip8pt 

(the {\it special} case) where $E_\sL$ is an extension of a representation $W(\delta_1,\delta_2)$ on a finite-dimensional $K$-vector space by ${\rm St}^{an}(\delta_1,\delta_2)$. Here, $W(\delta_1,\delta_2)$ is in fact a subrepresentation of $B^{an}(\delta_1,\delta_2)$ and ${\rm St}^{an}(\delta_1,\delta_2)$ denotes the corresponding quotient of $B^{an}(\delta_1,\delta_2)$.
We have $\dG B^{an}(\eta_1,\eta_2)=1$ according to Cor. \ref{cor-dimVERMA} for any pair of continuous characters $(\eta_1,\eta_2)$ which settles the generic case. Since $\dG W(\delta_1,\delta_2)=0$ we have $\dG E_\sL=1$ and this settles the special case.
\end{proof}

The preceding proposition suggests the following question: Are there any absolutely irreducible $2$-dimensional representations $V$ of $\cG_{\bbQ_p}$ such that $\dG \Pi(V)=2$?

\bibliographystyle{plain}
\bibliography{mybib}

\begin{thebibliography}{10}

\bibitem{ArdakovICM}
K.~Ardakov.
\newblock {$\hat{\mathcal D}$}-modules on rigid analytic spaces.
\newblock In {\em Proc. {I}nternational {C}ongress of {M}athematicians 2014 (to
  appear).}

\bibitem{AW}
Konstantin Ardakov and Simon Wadsley.
\newblock On irreducible representations of compact {$p$}-adic analytic groups.
\newblock {\em Ann. of Math. (2)}, 178(2):453--557, 2013.

\bibitem{BB81}
Alexandre Be{\u\i}linson and Joseph Bernstein.
\newblock Localisation de {$\mathfrak{g}$}-modules.
\newblock {\em C. R. Acad. Sci. Paris S\'er. I Math.}, 292(1):15--18, 1981.

\bibitem{BGG2}
I.~N. Bernstein, I.~M. Gelfand, and S.~I. Gelfand.
\newblock A certain category of {${\mathfrak g}$}-modules.
\newblock {\em Funkcional. Anal. i Prilo\v zen.}, 10(2):1--8, 1976.

\bibitem{BerthelotDI}
P.~Berthelot.
\newblock D-modules arithm\'etiques {I}. {O}p\'erateurs diff\'erentiels de
  niveau fini.
\newblock {\em {A}nn. {S}ci. {E}.{N}.{S}}, 29:185--272, 1996.

\bibitem{BK81}
J.-L. Brylinski and M.~Kashiwara.
\newblock Kazhdan-{L}usztig conjecture and holonomic systems.
\newblock {\em Invent. Math.}, 64(3):387--410, 1981.

\bibitem{CollingwoodMcGovern}
David~H. Collingwood and William~M. McGovern.
\newblock {\em Nilpotent orbits in semisimple {L}ie algebras}.
\newblock Van Nostrand Reinhold Mathematics Series. New York, 1993.

\bibitem{Colmez_LAvec}
Pierre Colmez.
\newblock La s\'erie principale unitaire de {${\rm GL}\sb 2(\bold Q\sb p)$}:
  vecteurs localement analytiques.
\newblock {\em Preprint 2014. {A}vailable at:
  {$https://www.imj-prg.fr/~pierre.colmez/locan.pdf$}}.

\bibitem{Colmez_trianguline}
Pierre Colmez.
\newblock La s\'erie principale unitaire de {${\rm GL}\sb 2(\bold Q\sb p)$}.
\newblock {\em Ast\'erisque}, (330):213--262, 2010.

\bibitem{Colmez10}
Pierre Colmez.
\newblock Repr\'esentations de {${\rm GL}\sb 2(\bold Q\sb p)$} et
  {$(\phi,\Gamma)$}-modules.
\newblock {\em Ast\'erisque}, (330):281--509, 2010.

\bibitem{DospSchraen}
Gabriel Dospinescu and Benjamin Schraen.
\newblock Endomorphism algebras of admissible {$p$}-adic representations of
  {$p$}-adic {L}ie groups.
\newblock {\em Represent. Theory}, 17:237--246, 2013.

\bibitem{EmertonA}
M.~Emerton.
\newblock Locally analytic vectors in representations of locally $p$-adic
  analytic groups.
\newblock {\em Preprint. To appear in: Memoirs of the AMS}.

\bibitem{Helemskii}
A.Ya. Helemskii.
\newblock {\em Banach and {L}ocally {C}onvex {A}lgebras}.
\newblock Oxford Science Publications. Oxford University Press, 1993.

\bibitem{H2}
James~E. Humphreys.
\newblock {\em Introduction to {L}ie algebras and representation theory},
  volume~9 of {\em Graduate Texts in Mathematics}.
\newblock Springer-Verlag, New York, 1978.
\newblock Second printing, revised.

\bibitem{H1}
James~E. Humphreys.
\newblock {\em Representations of semisimple {L}ie algebras in the {BGG}
  category {$\mathcal{O}$}}, volume~94 of {\em Graduate Studies in
  Mathematics}.
\newblock American Mathematical Society, Providence, RI, 2008.

\bibitem{Jantzen_Einhuellende}
Jens~Carsten Jantzen.
\newblock {\em Einh\"ullende {A}lgebren halbeinfacher {L}ie-{A}lgebren},
  volume~3 of {\em Ergebnisse der Mathematik und ihrer Grenzgebiete (3)
  [Results in Mathematics and Related Areas (3)]}.
\newblock Springer-Verlag, Berlin, 1983.

\bibitem{JantzenCharp}
Jens~Carsten Jantzen.
\newblock Representations of {L}ie algebras in prime characteristic.
\newblock In {\em Representation theories and algebraic geometry ({M}ontreal,
  {PQ}, 1997)}, volume 514 of {\em NATO Adv. Sci. Inst. Ser. C Math. Phys.
  Sci.}, pages 185--235. Kluwer Acad. Publ., Dordrecht, 1998.
\newblock Notes by Iain Gordon.

\bibitem{KL_CoxeterHecke}
David Kazhdan and George Lusztig.
\newblock Representations of {C}oxeter groups and {H}ecke algebras.
\newblock {\em Invent. Math.}, 53(2):165--184, 1979.

\bibitem{Lepowsky}
J.~Lepowsky.
\newblock Generalized {V}erma modules, the {C}artan-{H}elgason theorem, and the
  {H}arish-{C}handra homomorphism.
\newblock {\em J. Algebra}, 49(2):470--495, 1977.

\bibitem{Levasseur}
Thierry Levasseur.
\newblock Some properties of noncommutative regular graded rings.
\newblock {\em Glasgow Math. J.}, 34(3):277--300, 1992.

\bibitem{LVO}
H.~Li and F.~van Oystaeyen.
\newblock {\em Zariskian filtrations}, volume~2 of {\em $K$-Monographs in
  Mathematics}.
\newblock Kluwer Academic Publishers, Dordrecht, 1996.

\bibitem{LiHuishi}
Huishi Li.
\newblock Lifting {O}re sets of {N}oetherian filtered rings and applications.
\newblock {\em J. Algebra}, 179(3):686--703, 1996.

\bibitem{LXZ_LAvec}
Ruochuan Liu, Bingyong Xie, and Yuancao Zhang.
\newblock Locally analytic vectors of unitary principal series of {${\rm GL}\sb
  2(\Bbb Q\sb p)$}.
\newblock {\em Ann. Sci. \'Ec. Norm. Sup\'er. (4)}, 45(1):167--190, 2012.

\bibitem{MCR}
J.~C. McConnell and J.~C. Robson.
\newblock {\em Noncommutative {N}oetherian rings}.
\newblock Pure and Applied Mathematics (New York). John Wiley \& Sons Ltd.,
  Chichester, 1987.

\bibitem{Orlik08}
S.~Orlik.
\newblock Equivariant vector bundles on {D}rinfeld's upper half space.
\newblock {\em Invent. Math.}, 172(3):585--656, 2008.

\bibitem{OrlikStrauchJH}
S.~Orlik and M.~Strauch.
\newblock On {J}ordan-{H}\"older series of some locally analytic
  representations.
\newblock {\em Preprint 2010. {T}o appear in: {J}ournal of the {A}{M}{S}}.

\bibitem{Rinehart}
George~S. Rinehart.
\newblock Differential forms on general commutative algebras.
\newblock {\em Trans. Amer. Math. Soc.}, 108:195--222, 1963.

\bibitem{SchmidtAUS}
T.~Schmidt.
\newblock Auslander regularity of {$p$}-adic distribution algebras.
\newblock {\em Represent. Theory}, 12:37--57, 2008.

\bibitem{SchmidtSTAB}
T.~Schmidt.
\newblock Stable flatness of nonarchimedean hyperenveloping algebras.
\newblock {\em Journal of Algebra}, 323/3:757--765, 2010.

\bibitem{SchmidtBGG}
T.~Schmidt.
\newblock Verma modules over {$p$}-adic {A}rens-{M}ichael envelopes of
  reductive {L}ie algebras.
\newblock {\em {J}ournal of {A}lgebra}, 390:160--180, 2013.

\bibitem{SchmidtDIM}
Tobias Schmidt.
\newblock On locally analytic {B}eilinson-{B}ernstein localization and the
  canonical dimension.
\newblock {\em Math. Z.}, 275(3-4):793--833, 2013.

\bibitem{NFA}
P.~Schneider.
\newblock {\em Nonarchimedean functional analysis}.
\newblock Springer Monographs in Mathematics. Springer-Verlag, Berlin, 2002.

\bibitem{ST02a}
P.~Schneider and J.~Teitelbaum.
\newblock Banach space representations and {I}wasawa theory.
\newblock {\em Israel J. Math.}, 127:359--380, 2002.

\bibitem{ST02b}
P.~Schneider and J.~Teitelbaum.
\newblock Locally analytic distributions and {$p$}-adic representation theory,
  with applications to {${\rm GL}\sb 2$}.
\newblock {\em J. Amer. Math. Soc.}, 15(2):443--468 (electronic), 2002.

\bibitem{ST03}
P.~Schneider and J.~Teitelbaum.
\newblock Algebras of {$p$}-adic distributions and admissible representations.
\newblock {\em Invent. Math.}, 153(1):145--196, 2003.

\bibitem{ST05}
P.~Schneider and J.~Teitelbaum.
\newblock Duality for admissible locally analytic representations.
\newblock {\em Represent. Theory}, 9:297--326 (electronic), 2005.

\bibitem{Venjakob02}
Otmar Venjakob.
\newblock On the structure theory of the {I}wasawa algebra of a {$p$}-adic
  {L}ie group.
\newblock {\em J. Eur. Math. Soc. (JEMS)}, 4(3):271--311, 2002.

\end{thebibliography}

\end{document}